\documentclass[11pt,a4paper]{amsart}
\usepackage{tikz-cd}
\usepackage[margin=1in]{geometry}  % set the margins to 1in on all sides
\usepackage{graphicx}              % to include figures
\usepackage{amsmath,euscript,comment}               % great math stuff
\usepackage{amsfonts}              % for blackboard bold, etc
\usepackage{amssymb,amscd,epsf,amsthm, epsfig}  

\usepackage{enumerate}              % better theorem environments
\usepackage[toc,page]{appendix}
\usepackage{amsmath,calligra,mathrsfs}
\usepackage{float}
\usepackage[symbol]{footmisc}
\usepackage{tikz-cd}
\usepackage{xcolor}

\newtheorem{theorem}{Theorem}[section]
\newtheorem{lemma}[theorem]{Lemma}
\newtheorem{proposition}[theorem]{Proposition}
\newtheorem{corollary}[theorem]{Corollary}

\newtheorem{lemma-definition}[theorem]{Lemma-Definition}

\theoremstyle{definition}
\newtheorem{definition}[theorem]{Definition}

\newtheorem{remark}[theorem]{Remark}

\numberwithin{theorem}{section}

\newcommand{\Z}{\mathbb{Z}}

\newcommand{\calC}{\mathcal{C}}
\newcommand{\calD}{\mathcal{D}}
\newcommand{\calE}{\mathcal{E}}

\DeclareMathOperator{\HOM}{\mathscr{H}\text{\kern -3pt {\calligra\large om}}\,}

\newcommand{\DD}{\EuScript D}

\newcommand{\KK}{\EuScript K}

\newcommand{\sA}{\mathop{s\overline{A}}\nolimits}

\newcommand{\Hom}{\mathop{\mathrm{Hom}}\nolimits}

\newcommand{\Barr}{\mathop{\mathrm{Bar}}\nolimits}

\newcommand{\id}{\mathop{\mathrm{id}}\nolimits}

\newcommand{\op}{\mathop{\mathrm{op}}\nolimits}

\newcommand{\sg}{\mathop{\mathrm{sg}}\nolimits}

\newcommand{\nc}{\mathop{\mathrm{nc}}\nolimits}

\newcommand{\HH}{\mathop{\mathrm{HH}}\nolimits}
\newcommand{\THH}{\mathop{\mathrm{TH}}\nolimits}

\newcommand{\Perf}{\mathop{\mathrm{Perf}}\nolimits}

\input xy
\xyoption{all}

\usepackage{lipsum}

\usepackage{hyperref}

\begin{document}

\title{Invariance of the Goresky-Hingston algebra on reduced Hochschild homology}
%\author{Manuel Rivera and Zhengfang Wang}
%\author[1]{Manuel Rivera \thanks{manuel,rivera@imj-prg.fr}}
%\author[2]{Zhengfang Wang\thanks{zhengfang.wang@imj-prg.fr}}
%\curraddr{CNRS, UMR 7586, Institut de Math\'{e}matiques de Jussieu-Paris Rive Gauche, Universit\'{e} Pierre et Marie Curie, Case 247, Bureau 15-25/501, 4 place Jussieu, F-75005, Paris, France}
%\email{manuel.rivera@imj-prg.fr}
%\footnote{a}
%\thanks{The first author is currently supported by the ERC via the Starting Grant ERC StG-259118 managed by the CNRS}
%\affil[1]{}
%\affil[2]{}
\author{Manuel Rivera \and Zhengfang Wang}

\newcommand{\Addresses}{{% additional braces for segregating \footnotesize
  \bigskip
  \footnotesize

  Z.~Wang, \textsc{Max Planck Institute for Mathematics,  Vivatsgasse 7,  53111 Bonn Germany; \\
  Institute of Algebra and Number Theory, University of Stuttgart, Pfaffenwaldring 57, 70569 Stuttgart Germany 
}
\par\nopagebreak
  \textit{E-mail address}: \texttt{zhengfangw@gmail.com}

  \medskip

  M.~Rivera 
  (Corresponding author),
  \textsc{Purdue University, Department of Mathematics, 150 N. University St. West Lafayette, IN 47907}\par\nopagebreak
  \textit{E-mail address}: \texttt{manuelr@purdue.edu}

 }}

%\affil[$\dagger$]{}

%\email{manuelr@purdue.edu}

%\email{zhengfangw@gmail.com}

\maketitle

\begin{abstract}
We prove that two quasi-isomorphic simply connected differential graded associative  Frobenius algebras have isomorphic Goresky-Hingston algebras on their reduced Hochschild homology. Our proof is based on relating the Goresky-Hingston algebra on reduced Hochschild homology to the singular Hochschild cohomology algebra. For any simply connected oriented closed manifold $M$ of dimension $k$, the Goresky-Hingston algebra on reduced Hochschild homology induces an algebra structure of degree $k-1$ on $\overline{\mathrm H}^*(LM;\mathbb{Q})$, the reduced rational cohomology of the free loop space of $M$. As a consequence of our algebraic result, we deduce that the isomorphism class of the induced algebra structure on  $\overline{\mathrm H}^*(LM;\mathbb{Q})$ is an invariant of the homotopy type of $M$. 

\end{abstract}

\section{Introduction}

\subsection{String topology and motivation} String topology is concerned with the algebraic structure of the free loop space $LM$ of a manifold $M$. The field began with the construction of a graded associative and commutative product on the homology of $LM$ defined by combining the intersection product on the underlying manifold $M$ with the concatenation product on pairs of loops in $M$ with a common base point \cite{ChSu}. Later on, a graded coassociative and cocommutative coproduct on the homology of $LM$ relative to constant loops $M \subset LM$ was described in \cite{Sul} and \cite{GoHi} by considering self intersections in a single family of loops and cutting at these intersection points to obtain two new families. These two operations are part of a rich family of compatible operations, which may be constructed at the chain level by intersecting, cutting, and reconnecting families of loops according to combinatorial patterns associated to moduli spaces of surfaces (e.g. \cite{Sul}, \cite{DrPoRo}).

There are different ways of making choices to construct string topology operations rigorously. Some use geometric and topological methods to describe intersections (\cite{GoHi}, \cite{HiWa},\cite{DrPoRo}, \cite{Sul}), others homotopy theoretic techniques \cite{CoJo}, and others start with an algebraic chain or cochain model for a manifold, such as the commutative dg  algebra of differential forms, and then make choices algebraically in order to construct operations using the relationship between Hochschild homology and the free loop space (\cite{WaWe}, \cite{Abb}, \cite{TrZe}, \cite{Kau}, \cite{NaWi}). This article is concerned with an algebra structure defined through this last approach on a reduced version of the Hochschild homology of a differential graded Frobenius algebra over a field of characteristic zero. Our main result is a purely algebraic statement regarding the quasi-isomorphism invariance of this algebra structure under a simply connected hypothesis on the underlying dg Frobenius algebra. The algebraic product studied in this article resembles a geometric construction known as the Goresky-Hingston coproduct (\cite{Sul}, \cite{GoHi}, \cite{HiWa}), which has been computed to be non-trivial in specific cases. The question of whether this operation is a homotopy invariant has been proven to be particularly subtle (\cite{HiWa2}, \cite{Na}). 

\subsection{Main theorem} Fix a field $\mathbb{K}$ of characteristic zero and write $\otimes = \otimes_{\mathbb{K}}$. Let $A$ be a unital dg $\mathbb{K}$-algebra  equipped with a non-degenerate symmetric pairing $\langle-, -\rangle \colon A \otimes A \to \mathbb{K}$ of degree $k>0$, which is compatible with the product and the differential of $A$. Such an object is called a \textit{dg  Frobenius algebra of degree $k$} (see Definition \ref{definition2.1}). In particular, any dg  Frobenius algebra of degree $k$ comes equipped with special degree $k$ element $\sum_{i} e_i \otimes f_i \in A \otimes A$ (called the Casimir element) satisfying $\sum_i xe_i \otimes f_i = \sum_i (-1)^{k|x|} e_i \otimes f_ix$ and $\sum_i(-1)^{|e_i||x|} e_i x \otimes f_i = \sum_i (-1)^{|e_i||x|} e_i \otimes x f_i$ for any $x \in A$. If we think of $A$ as a cochain model for a closed manifold $M$ of dimension $k$ and $\langle -, -\rangle\colon A \otimes A \to \mathbb{K}$ as a pairing inducing the Poincar\'e duality pairing then $\sum_i e_i \otimes f_i$ plays the role of a representative for the Thom class of the diagonal embedding $M \hookrightarrow M \times M$. 

Suppose $A$ is connected (i.e.\ $A^0 \cong \mathbb{K}$) and consider the double complex of Hochschild chains $C_{*,*}(A,A)$ where $C_{-p,n}(A,A) = ((s\overline{A})^{\otimes p} \otimes A)^{n}$, $\overline{A} \subset A$ denotes the positive degree elements of $A$, and for any graded vector space $V= \bigoplus_{j \in \mathbb{Z}} V^j$, $s^iV$ denotes the $i$-th shifted graded vector space given by $(s^iV)^j=V^{i+j}$ (we write $s=s^1$). Denote by $C_n(A,A)= \bigoplus_{p} C_{-p,n}(A,A)$ and $C_*(A,A) = \bigoplus_n C_n(A,A)$. In this article we study a product $$\star \colon  C_*(A,A) \otimes C_*(A,A) \to C_*(A,A)$$ of degree $k-1$ defined by the formula (see Definition \ref{definition:GH})
$$(\overline{a_1} \otimes \dotsb \otimes \overline{a_p} \otimes a_{p+1}) \star (\overline{b_1} \otimes \dotsb \otimes \overline{b_q} \otimes b_{q+1}) = \sum_i (-1)^{\eta_i}\overline{b_1}\otimes \cdots \otimes \overline{b_{q+1}e_i}\otimes \overline{a_1}\otimes \cdots \otimes \overline{a_p} \otimes a_{q+1}f_i.$$

The above product appears as a secondary operation in the family of algebraic constructions on the Hochschild chain complex of a Frobenius algebra constructed in \cite{TrZe, Kau1, WaWe}. In the context when the underlying dg Frobenius is commutative, the product $\star$ has been studied in \cite{Abb} and \cite{Kla}. In Section \ref{section2}, we observe that the product $\star$ defines a (non-unital) dg algebra structure on the (shifted) \textit{reduced Hochschild complex} $s^{1-k} \overline{C}_*(A,A)$ of any connected  (not necessarily commutative) dg Frobenius algebra of degree $k$. The reduced Hochschild complex is the subcomplex of $C_*(A,A)$ defined as the complement of $C_{0,0}(A,A)=A^0 \cong \mathbb{K}$. We call the induced algebra structure on homology $(s^{1-k}\overline{\HH}_*(A,A),\star)$ the \textit{Goresky-Hingston algebra on the reduced Hochschild homology of a dg Frobenius algebra $A$}. The reason for this name is the analogy and similarities between the properties of the algebraic product $\star$ and the geometrically defined operation of \cite{GoHi} of the same degree. Recently, this analogy has been made mathematically precise: it has been announced in \cite{NaWi} that if $A$ is a Poincar\'e duality model for the rational polynomial differential forms on a simply connected oriented closed manifold $M$ (as constructed in \cite{LaSt}) then the product $\star$ at the level of a relative version of Hochschild homology of $A$ corresponds to a topologically defined version of the Goresky-Hingston product on $\mathrm H^*(LM,M; \mathbb{Q})$, the rational cohomology of the free loop space $LM$ relative to constant loops. 

In this article we give an algebraic proof of the invariance of the isomorphism class of the algebra $(s^{1-k}\overline{\HH}_*(A,A),\star)$ under quasi-isomorphisms of simply connected dg Frobenius algebras in the following sense. Let $\mathbf{DGA}^1_{\mathbb{K}}$ be the category of  unital dg  $\mathbb{K}$-algebras $A$ which are simply connected and non-negatively graded, i.e.\ $A^{<0} =0, \ A^0 \cong \mathbb{K}$ and $A^1=0$. 

\begin{theorem} \label{main} Let $(A, \langle -, -\rangle_A)$ and $(B, \langle -, -\rangle_B)$ be two dg  Frobenius algebras of degree $k$ such that $A,B \in \mathbf{DGA}^1_{\mathbb{K}}$. Suppose that there is a zig-zag of quasi-isomorphisms of dg algebras $$A \xleftarrow{\simeq} \bullet \xrightarrow{\simeq} \dotsb \xleftarrow{\simeq} \bullet \xrightarrow{\simeq} B.$$ Then there is an isomorphism of Goresky-Hingston algebras $$(s^{1-k}\overline{\HH}_*(A,A), \star)  \cong (s^{1-k}\overline{\HH}_*(B,B), \star).$$
\end{theorem} 

In the above statement there is no required compatibility between the quasi-isomorphisms in the zig-zag and the pairings $\langle -, -\rangle_A$ and $\langle -, -\rangle_B$. However, the isomorphism we construct between the Goresky-Hingston algebras of $A$ and $B$ uses both pairings $\langle -, -\rangle_A$ and $\langle -, -\rangle_B$ as well as the quasi-isomorphisms in the zig-zag.

\subsection{Outline of the proof} 

One of the main purposes of this article is to highlight the techniques used in the proof of the above theorem which rely on  the invariance properties of the \textit{singular Hochschild cochain complex}, as introduced in the second  author's thesis and in \cite{Wan} and extended in \cite{RiWa} to the dg setting. Given any dg algebra $A$ (no Frobenius structure required) the singular Hochschild cochain complex $\calC_{\sg}^*(A,A)$ is a dg algebra, with a cup product $\cup$ extending the classical cup product on Hochschild cochains (see Subsection \ref{subsection:cupproduct}). We denote the cohomology of $\calC_{\sg}^*(A,A)$ by $\HH_{\sg}^*(A, A)$ and call it the {\it singular Hochschild cohomology of $A$} (also called {\it Tate-Hochschild cohomology}). Under mild finiteness conditions on $A$,  $\HH_{\sg}^*(A, A)$  is the graded algebra of morphisms from $A$ to itself in the {\it singularity category} $$\DD_{\sg}(A\otimes A^{\op})= \DD^b(A \otimes A^{\op}) / \Perf(A\otimes A^{\op}),$$ i.e.\ the Verdier quotient of the bounded derived category of finitely generated dg $A$-$A$-bimodules  by the full subcategory of perfect dg $A$-$A$-bimodules (\cite{Wan}, \cite{RiWa}). The singularity category was introduced in \cite{Buc} and used in \cite{Orl} to study singularities of algebraic varieties.

The proof of our main theorem may be outlined in three steps.

{\it Step 1:} The isomorphism class of the graded algebra $\HH_{\sg}^*(A, A)$ is invariant under quasi-isomorphisms (see 
Proposition \ref{proposition-7.4}). Namely, if $A$ and $B$ are quasi-isomorphic dg algebras then there is an isomorphism of graded algebras $\HH_{\sg}^*(A, A) \cong \HH_{\sg}^*(B, B).$

{\it Step 2:} When the dg algebra $A$ is equipped with a dg Frobenius structure of degree $k$, there is a smaller complex $\mathcal{D}^*(A,A)$, called the Tate-Hochschild complex of $A$, that also computes $\HH_{\sg}^*(A, A)$ (see Definition \ref{definition:TH}). The complex $\mathcal{D}^*(A,A)$ is defined as the mapping cone of a chain map (see \eqref{mappingcone}) $$\gamma: s^{-k}C_*(A,A) \to C^*(A,A)$$ from the (shifted) Hochschild chain complex of $A$ to the Hochschild cochain complex of $A$. The map $\gamma$ is defined using the Frobenius structure of $A$. The Tate-Hochschild complex  $\mathcal{D}^*(A,A)$ carries a natural product $\star$ extending  the Goresky-Hingston product on $C_*(A,A)$, also denoted by $\star$ above, and the classical Hochschild cup product on $C^*(A,A)$. In fact, the Goresky-Hingston algebra $(s^{1-k}\overline{\HH}_*(A,A), \star)$ is a subalgebra of the cohomology algebra $(\mathrm{H}^*(\mathcal{D}^*(A,A)), \star)$. We  temporarily denote $(\mathrm{H}^*(\mathcal{D}^*(A,A)), \star)$ by $(\THH^*(A,A), \star)$. The isomorphism of graded algebras $$(\HH_{\sg}^*(A, A), \cup) \cong (\THH^*(A,A), \star),$$ which follows from Theorem \ref{thm:homtopyretract} and Proposition \ref{prop:iotastar}, implies that $(\THH^*(A,A), \star)$ is also invariant under quasi-isomorphisms. %The Tate-Hochschild complex of a dg Frobenius algebra $A$, denoted by $\mathcal{D}^*(A,A)$, is defined as the mapping cone of a chain map $\gamma: C_*(A,A) \to C^*(A,A)$ from the Hochschild chain complex of $A$ to the Hochschild cochain complex of $A$. The map $\gamma$ is defined using the Frobenius structure of $A$. The product $\star$ on the Tate-Hochschild complex extends the Goresky-Hingston product on $C_*(A,A)$, also denoted by $\star$ above, and the Hochschild cup product on $C^*(A,A)$. In fact, the Goresky-Hingston algebra $(s^{1-k}\overline{\HH}_*(A,A), \star)$ is a subalgebra of $(\THH^*(A,A), \star)$.

{\it Step 3:} To conclude the desired invariance for $(s^{1-k}\overline{\HH}_*(A,A), \star)$ from the invariance property of $(\THH^*(A,A), \star)$, we must show the following: if $A$ and $B$ are simply connected dg Frobenius algebras that are quasi-isomorphic as dg algebras, then the isomorphism $$(\THH^*(A,A), \star) \cong (\THH^*(B,B), \star)$$ restricts to an isomorphism between subalgebras  $$(s^{1-k}\overline{\HH}_*(A,A), \star) \cong (s^{1-k}\overline{\HH}_*(B,B), \star).$$ To show this we consider two cases: when the Euler characteristic of the underlying Frobenius algebras is non-zero or zero.  In the first case, when the Euler characteristic of $A$ (and hence of $B$) is non-zero,  the result follows immediately from our description of $\THH^*(A,A)$ in terms of $\HH^*(A,A)$ and $\HH_*(A,A)$ (see Proposition \ref{proposition4.7}). Namely in this case $s^{1-k}\overline{\HH}_*(A,A)$ (resp.\ $s^{1-k}\overline{\HH}_*(B, B)$) coincides with the truncation $\THH^{*\geq k}(A, A)$ (resp.\ $\THH^{*\geq k}(B, B)$) and note that $\THH^{*\geq k}(A, A) \cong \THH^{*\geq k}(B, B)$ as graded algebras.  The second case, when the Euler characteristic of $A$ (and hence of $B$) is zero, is more subtle since by Proposition \ref{proposition4.7} we have $s^{1-k}\overline{\HH}_*(A,A) \oplus\HH^k(A,A)\cong \THH^{*\geq k}(A,A)$. This leads us to take a closer look at the functoriality properties of the quasi-isomorphism $$\iota \colon \mathcal{D}^*(A,A) \to \calC_{\sg}^*(A,A)$$ constructed in  Theorem \ref{thm:homtopyretract}. In Proposition \ref{factor}, we factor this quasi-isomorphism as the composition of two quasi-isomorphisms $$\mathcal{D}^*(A,A) \to \mathcal{E}^*(A,A) \to  \calC_{\sg}^*(A,A),$$ where $\mathcal{E}^*(A,A)$ is a new complex constructed using the inverse dualizing complex of $A$. The complex $\mathcal{E}^*(A,A)$ has a better functorial behavior (see Proposition \ref{functoriality} and Corollary \ref{corollary:longexactsequence}) that allows us to conclude the desired result, namely the algebra isomorphism $\THH^{* \geq k}(A, A) \cong THH^{* \geq k }(B, B)$ restricts to $s^{1-k}\overline{\HH}_*(A,A) \cong s^{1-k}\overline{\HH}_*(B, B).$ 

In this article we use $\mathcal{E}^*(A,A)$ merely to treat the Euler characteristic zero case in the proof of our main result. The main advantage of $\mathcal{E}^*(A,A)$ is that it shares similarities with both the Tate-Hochschild complex and the singular Hochschild complex, i.e. $\mathcal{E}^*(A,A)$ is constructed as a mapping cone, like $\mathcal{D}^*(A,A)$, but without using a Frobenius structure, like $\calC_{\sg}^*(A,A)$. In fact, it is plausible that one may develop the whole theory of singular Hochschild cohomology at the chain level entirely in terms of the complex $\mathcal{E}^*(A,A)$  without alluding to $\calC_{\sg}^*(A,A)$. This will be explored elsewhere.

\subsection{Application to string topology}
Let $\mathcal{A}$ be a commutative dg (cdg) algebra  whose cohomology $\mathrm H^*(\mathcal{A})$ is a simply connected graded Frobenius algebra of degree $k$. By the main result of \cite{LaSt}, there exists a cdg Frobenius algebra $(A, \langle -, -\rangle_A)$ such that $A^0\cong \mathbb{K}$, $A^1=0$, and $A$ is quasi-isomorphic to $\mathcal{A}$ through a zig-zag of cdg algebras such that the induced isomorphism on cohomology preserves the graded Frobenius algebra structure. Following \cite{LaSt}, we call $A$ a \textit{Poincar\'e duality (cdg) model} for $\mathcal{A}$. For any simply connected oriented closed manifold $M$ of dimension $k$, the Goresky-Hingston algebra on reduced Hochschild homology induces an algebra structure on $s^{1-k}\overline{\mathrm H}^*(LM;\mathbb{Q})$, the shifted reduced rational cohomology of the free loop space on $M$, by choosing a Poincr\'e duality cdg model $A$ for the cdg $\mathbb{Q}$-algebra $\mathcal{A}(M)$ of rational polynomial differential forms on $M$ and using the isomorphisms of graded vector spaces  $$\overline{\mathrm H}^*(LM;\mathbb{Q} ) \  \cong \  \overline{\HH}_*(\mathcal{A}(M), \mathcal{A}(M)) \ \cong \  \overline{\HH}_*(A,A).$$
The first isomorphism above is induced by the Chen's classical iterated integrals construction or by a well known result of J.D.S. Jones. The second isomorphism follows from the quasi-isomorphism invariance of Hochschild homology. As an immediate consequence of our main theorem we have the following result.

\begin{corollary}
\label{corollary1.2}
\begin{enumerate} 
\item Let $M$  be a simply connected oriented closed manifold of dimension $k$ and $A$ a Poincar\'e duality  model for the cdg algebra of rational differential forms $\mathcal{A}(M)$. The isomorphism class of the algebra structure on $s^{1-k}\overline{\mathrm H}^*(LM;\mathbb{Q})$ induced by the product $\star$ through the isomorphism $\overline{\mathrm H}^*(LM;\mathbb{Q} ) \cong \overline{\HH}_*(A,A)$ is independent of the choice of Poincar\'e duality model $A$ for $\mathcal{A}(M)$. 
\item If $M$ and $M'$ are homotopy equivalent simply connected oriented closed manifolds of dimension $k$, then the algebra structures on $s^{1-k} \overline{\mathrm H}^*(LM;\mathbb{Q})$ and $s^{1-k}\overline{\mathrm H}^*(LM';\mathbb{Q})$ are isomorphic. 
\end{enumerate}
\end{corollary}

\subsection{Related results in the literature} 

The geometric constructions for the Goresky-Hingston operation describe an operation defined at the level at the homology (or chains) of the free loop space relative to constant loops, as described in (\cite{Sul}, \cite{GoHi}, \cite{HiWa}, \cite{NaWi}). The identification between the algebraic product on the reduced Hochschild homology of a Poincar\'e duality cdga model for a simply connected manifold $M$ and a geometric construction of the product on  $\mathrm H^*(LM,M)$ using configuration spaces of two points has been announced in \cite{NaWi}. Hingston and Wahl have announced in \cite{HiWa2} a version of homotopy invariance for the geometric Goresky-Hingston coproduct on $\mathrm H_*(LM,M)$ (with arbitrary coefficients) for certain type of homotopy equivalences satisfying suitable  restrictions. In the non-simply connected case, a counterexample for the homotopy invariance of the geometric Goresky-Hingston coproduct has been announced in \cite{Na}. Naef shows that the coproduct can distinguish homotopy equivalent Lens spaces $L(1,7)$ and $L(2,7)$. 

Proposition \ref{proposition4.7}, together with \cite[Theorem 1.10]{CiFrOa}, yields that the singular Hochschild cohomology of the dg algebra of cochains (with real coefficients) on a simply connected oriented closed manifold $M$ is isomorphic to the {\it Rabinowitz-Floer homology} of the unit cotangent bundle of $M$. We expect that the chain level algebraic structure of the Tate-Hochschild complex  may provide a better understanding of the structure of the geometric operations on Rabinowitz-Floer theory. 

For the relationship between the singular Hochschild cochain complex of an algebra and the Hochschild cochain complex of the dg singularity category, we refer to the recent article \cite{Kel}. 

\subsection{Conventions.} Throughout this article we will work over a fixed field $\mathbb{K}$ of characteristic zero. The sign conventions are obtained from the Koszul sign rule: when $a$ moves past $b,$ a sign change of $(-1)^{|a||b|}$ is required.  We refer to the appendix of \cite{Abb} for more on sign conventions. By {\it dg algebras} we mean differential graded associative algebras over the field $\mathbb K$. 
\subsection{Acknowledgements.} The second author would like to thank Purdue University where several aspects of this article were discussed during a visit to the first author. Both authors would like to thank Bernhard Keller, Ralph Kaufmann, Florian Naef, Nathalie Wahl, and Thomas Willwacher for insightful discussions, comments, and email exchanges. We also thank the anonymous referees for their valuable comments and suggestions.

\section{Goresky-Hingston algebra on the reduced Hochschild homology of a connected dg Frobenius algebra}
\label{section2}
In this section we start by recalling some classical definitions and constructions. Then we describe the Goresky-Hingston algebra on the reduced Hochschild homology of a dg Frobenius algebra.

\subsection{Differential graded associative Frobenius algebras}
\label{section2.1} Let $\mathbb K$ be a field.  Recall that a {\it dg vector space}   $(V, d_V)$ is a graded $\mathbb K$-vector space $V = \bigoplus_{j\in \mathbb Z} V^j$ together with a graded $\mathbb K$-linear map $d_V \colon V \to V$ of degree one (i.e.\ $d_V(V^j) \subset V^{j+1}$) such that $d_V\circ d_V = 0.$ For any element $a \in V^j$ we write $|a| = j$.
 
The {\it dual} of a dg vector space $(V, d_V)$ is defined as the dg vector space $(V^{\vee}, d_{V^{\vee}})$ with $(V^{\vee})^{-j} = \Hom_{\mathbb K}(V^{j}, \mathbb K)$ and $d_{V^{\vee}}(\alpha)(x) = -(-1)^{|\alpha|} \alpha (d_V(x))$ for any homogeneous elements $\alpha \in V^{\vee}$ and $x \in V$.  Let $(U, d_U)$ and $(V, d_V)$ be two dg vector spaces. There is a natural inclusion of dg vector spaces
\begin{align}
\label{dual-isomorphism}
\sigma_{U, V}\colon U^{\vee} \otimes V^{\vee} \hookrightarrow (V \otimes U)^{\vee}, \quad \alpha \otimes \beta \mapsto \big(v \otimes u \mapsto \beta(v) \alpha(u)\big).
\end{align}
If either $U$ or $V$ is finite dimensional as a $\mathbb K$-vector space then the above inclusion becomes an isomorphism.  

A {\it dg algebra} $A=(A, d,  \mu)$ over $\mathbb K$ is a dg $\mathbb{K}$-vector space $(A,d)$ equipped with an associative product $\mu\colon A \otimes A \to A$ of degree zero which is a cochain map, i.e.\ $\mu$ satisfies the identity $\mu \circ (d \otimes \text{id} + \text{id} \otimes d)= d \circ \mu$. A dg algebra is said to be {\it unital} if there is a map $u\colon \mathbb{K} \to A$ such that $\mu \circ (u \otimes \text{id}) = \text{id} = \mu \circ (\text{id} \otimes u)$. Denote by $1_{A}$ or just by $1$ the image of $1_{\mathbb{K}}$ under $u$.  

\begin{definition}\label{definition2.1} Let $k$ be a positive integer.  A {\it dg Frobenius algebra of degree $k$} is a non-negatively graded dg algebra $(A,d, \mu)$ equipped with a pairing $\langle -, -\rangle\colon A \otimes A \to \mathbb{K}$ such that
\begin{enumerate}[(i)]
\item $\langle -, -\rangle$ is of degree $- k$, namely, it is possibly non-zero only on $A^i \otimes A^{k-i}$ for any $i=0,\dotsb,k$
\item $\langle -, -\rangle$ is non-degenerate, namely, the induced map 
$$\rho\colon A \to A^{\vee}, \quad a \mapsto (b \mapsto \langle a, b\rangle) 
$$
is an isomorphism of degree $-k$
\item $\langle \mu(a\otimes b),c\rangle = \langle a,\mu(b\otimes c)\rangle$ \quad for any $a,b,c \in A$
\item $\langle a,b\rangle = (-1)^{|a| |b|}\langle b,a\rangle $ \quad for any $a,b \in A$ 
\item $\langle d(a),b\rangle  = -(-1)^{|a|}\langle a,d(b)\rangle $ \quad  for any $a,b \in A$.
\end{enumerate}
\end{definition}

The non-degeneracy of the pairing implies that $A$ is finite dimensional as a $\mathbb{K}$-vector space and $A^i=0$ for $i>k$. It follows that the inclusion in (\ref{dual-isomorphism}) induces an isomorphism $\sigma_{A, A} \colon A^{\vee} \otimes A^{\vee} \xrightarrow{\cong} (A\otimes A)^{\vee}$.  
We will write $\mu(a\otimes b)=ab$. 

Conditions (iii)-(v) imply that $\rho\colon A \to A^{\vee}$ is a map of dg $A$-$A$-bimodules of degree $-k$. Recall the $A$-$A$-bimodule structure on $A^{\vee}$ is given by $$ 
(a\otimes b)\cdot \beta (c) = (-1)^{|\beta| (|a| + |b|) + |a| (|b| + |c|)} \beta(bca), \quad \text{for any $\beta \in A^{\vee}$ and $a, b, c \in A$.}
$$

Denote by $\Delta\colon A \to A \otimes A$ the coproduct of degree $k$ defined by the composition
$$A  \xrightarrow{ \rho} A^\vee \xrightarrow{\mu^{\vee}} (A\otimes A)^{\vee} \xrightarrow{\sigma_{A, A}^{-1}}  A^{\vee} \otimes A^{\vee} \xrightarrow{ (\rho \otimes \rho)^{-1}} A \otimes A.$$ Denote $\Delta(1):= \sum_{i } e_i \otimes f_i \in A \otimes A$. In the algebraic literature $\sum_{i } e_i \otimes f_i$ is called the \textit{Casimir element} of the Frobenius algebra $A$. If we think of $A$ as a cochain model for a manifold $M$, then we may think of the Casimir element as a representative for the Thom class of the diagonal embedding $M \hookrightarrow M \times M$.
\\
\\
The following identities regarding the Carimir element will be useful in our algebraic manipulations. 
%The following lemma will be very useful.  %Note $\text{deg}(e_i) + \text{deg}(f_i) =k$. Since $\Delta$ is a map of dg $A$-$A$-bimodules it follows that $\sum_i d(e_i)\otimes f_i=- \sum_i (-1)^{\deg(e_i)} e_i \otimes d(f_i) $ and $\sum_{i} xe_i \otimes f_i = \sum_{i }(-1)^{k\text{deg}(x)}e_i \otimes f_ix$ for any $x \in A$. Similarly we have $ \sum_i e_ix \otimes f_i = e_i \otimes xf_i$ since $\Delta$ is also a right dg $A$-$A$-bimodules, i.e.\ $\Delta(x) =  \Delta(1)\cdot (1\otimes x)=  \Delta(1)\cdot(x \otimes 1).$ The dot symbol $\cdot$ indicates the multiplication of $A\otimes A^{\op}$,  i.e.\ for any $a_1,a_2,b_1,b_2 \in A$ we have $(a_1 \otimes b_1)\cdot (a_2 \otimes b_2) = (-1)^{|b_1|(|a_2| + |b_2|)}a_1 a_2 \otimes b_2b_1$.

\begin{lemma}\label{lemma:bascipropertyeifi}
Let $a \in A$. We have the following identities
\begin{enumerate}
\item $(-1)^{|a| k} \sum_i 
 \langle f_i, a \rangle e_i =a = (-1)^{k-|a|}  \sum_i  \langle e_i, a \rangle f_i$
 \item $\sum_i e_i \otimes f_i = \sum_i(-1)^{|e_i||f_i|+k} f_i \otimes e_i$
\item $\sum_i de_i \otimes f_i = -\sum_i (-1)^{|e_i|} e_i \otimes df_i$
\item $\sum_i a e_i \otimes f_i = (-1)^{|a| k} \sum_i e_i\otimes f_i a$
\item $ \sum_i (-1)^{|e_i||a|} e_i a \otimes f_i = \sum_i (-1)^{|e_i| |a|} e_i \otimes a f_i$. 
%\item $ \sum_i (-1)^{|a|+|f_i|(|a_1|+\dotsb + |a_p|)}e_i a \otimes a_1 \otimes \dotsb \otimes a_p \otimes f_i =\sum (-1)^{(|a|+|f_i|)(|a_1|+\dotsb + |a_p|)} e_i \otimes a_1 \otimes \dotsb \otimes a_p \otimes af_i$
%\item $ \sum_i ae_i  \otimes a_1 \otimes \dotsb \otimes a_p \otimes f_i = (-1)^{|a|(|a_1|+\dotsb + |a_p|+k)} e_i \otimes a_1 \otimes \dotsb \otimes a_p \otimes f_ia $
%\item $\sum a f_i \otimes e_i = \sum (-1)^{|a|(k-1)} f_i \otimes e_i a$
%\item $ \sum f_i a \otimes e_i = \sum f_i \otimes ae_i$.
\end{enumerate}
\end{lemma}
\begin{proof}
By definition, we have 
$\sigma_{A, A}\circ (\rho\otimes \rho) \circ \Delta = \mu^{\vee} \circ \rho$. It follows that for any $a, b \in A$ 
\[
\sum_i (-1)^{|e_i|k} \langle f_i, a\rangle \langle e_i, b \rangle =\langle a, b\rangle. 
\]
%Equivalently,  we have
%\[
% \left \langle \sum_i (-1)^{|e_i|k}\langle f_i, a\rangle  e_i, b \right\rangle =\langle a, b\rangle \quad \text{and} \quad  \left \langle \sum_i (-1)^{|e_i|k} \langle e_i, b \rangle f_i, a\right\rangle  =\langle a, b\rangle.
%\]
By the non-degeneracy of $\langle -, - \rangle$ we obtain %$$\sum_i (-1)^{|e_i|k}\langle f_i, a\rangle  e_i = a\quad \text{and} \quad \sum_i (-1)^{|e_i|} \langle e_i, b \rangle f_i = b.$$
\[
(-1)^{|a| k} \sum_i \langle f_i, a\rangle e_i = a = (-1)^{k-|a|} \sum_i \langle e_i, a \rangle f_i,
\] 
where we implicitly  use the fact that $\langle f_i , a \rangle$ is zero if $|f_i|+|a|\neq k$.

The second assertion follows by noting
\[\small
\sum_{i} e_i \otimes f_i = \sum_{i, j} (-1)^{|f_i|k} e_i  \otimes   \langle f_j, f_i \rangle e_j  = \sum_{i, j} (-1)^{|f_i|k+|f_i||f_j| } \langle f_i, f_j \rangle e_i  \otimes    e_j  =  \sum_j (-1)^{|e_j||f_j| +k} f_j\otimes e_j,
\]
where the first and third equalities follow from the first assertion. Here, we  also use the fact that  $\langle f_i, f_j \rangle = 0$  if $|f_i|\neq |e_j|$. The similar argument yields the remaining assertions. 
%To prove the second assertion, let us consider the following isomorphism 
%[
%A \otimes A \xrightarrow{ \id \otimes \rho} A \otimes A^{\vee} \xrightarrow{\mathrm{ev}} \Hom(A, A)
%\]
%where $\mathrm{ev}(a \otimes \alpha) (b) =  \alpha(b) a$. By the first assertion, we have 
%\begin{align*}
%\sum_i (\id \otimes \rho) \circ \mathrm{ev} (e_i \otimes f_i) & = \sum_i (-1)^{|e_i||f_i| + k} (\id \otimes \rho) \circ \mathrm{ev}(f_i \otimes e_i). 
%\end{align*}
\end{proof}

\begin{remark}\label{remark:usefulidentities}
The following equations are useful in keeping track of the signs. Using the Koszul sign rule, Lemma \ref{lemma:bascipropertyeifi} yields the following identities in $A \otimes V \otimes A$ for  any dg vector space $V$
\begin{align*}
\sum_i (-1)^{|x||e_i|} a e_i \otimes x \otimes f_i & = \sum_i (-1)^{|x||e_i|+|a|(|x|+k)} e_i \otimes x \otimes f_i a\\
\sum_i (-1)^{|e_i|(|a|+|x|)} e_i a \otimes x \otimes f_i &= \sum_i (-1)^{|e_i|(|a|+ |x|)+|a||x|} e_i \otimes x \otimes af_i,
\end{align*}
 where $x\in V$. %These two identities will be frequently used in Appendix . 
%where $x\in (s\overline A)^{\otimes m}$ 
%Let us give more comments on the signs. In practice, 
%Note that $\langle f_i , a \rangle$ is possibly non-zero only when $|f_i|+|a|=k$. Hence, the first identity in Lemma 2.2 may also be written as 
%\[
%(-1)^{|a| k} \sum_i \langle f_i, a\rangle e_i = a = (-1)^{k-|a|} \sum_i \langle e_i, a \rangle f_i.
%\] 
\end{remark}

It follows from Lemma \ref{lemma:bascipropertyeifi}  that $\Delta\colon A \to A \otimes A$ is a map of dg $A$-$A$-bimodules of degree $k$. Lemma \ref{lemma:bascipropertyeifi} also implies that $\Delta$ defines a dg coassociative coalgebra structure on $A$ (of degree $k$) with counit $\varepsilon\colon A \to \mathbb{K}$, $\varepsilon(a)=\langle a,1\rangle $, where $1 \in A^0$ denotes the unit of $A$. Namely,  we have  $ (\Delta \otimes \id)\circ \Delta = (-1)^{k} (\id \otimes \Delta) \circ \Delta$ and $ (\id \otimes \epsilon) \circ \Delta = \id = (-1)^k (\epsilon \otimes \id) \circ \Delta$.

\begin{proposition}\label{remark:CY}
Any  dg Frobenius algebra of degree $k$ satisfies the \lq\lq $k$-Calabi-Yau'' property,  i.e. there is a quasi-isomorphism of dg $A$-$A$-bimodules between $s^{-k}A$ and  the {\it inverse dualizing complex} $A^!: = {\rm RHom}_{A^e}(A, A^e) $ introduced in \cite{VdB}.
\end{proposition}
\begin{proof}
Since $A$ is dg Frobenius of degree $k$, it follows that the enveloping algebra $A^e = A \otimes A^{\rm op}$ is dg Frobenius of degree $2k$ with pairing $\langle a_1 \otimes b_1, a_2 \otimes b_2\rangle := (-1)^{|a_2||b_1|} \langle a_1, a_2\rangle \langle b_1, b_2 \rangle$. In particular, it induces a natural isomorphism  $ s^{2k} A^e \cong(A^e)^{\vee}$ of right dg $A^e$-modules. Thus, we have the following natural quasi-isomorphisms of dg $A$-$A$-bimodules (i.e.\ dg left $A^e$-modules)
\begin{align}\label{align:leftrightmodules}
s^{-k} A  \cong {\rm RHom}_{\text{mod$-A^e$}}(s^{2k} A^e, s^k A)   \cong {\rm RHom}_{\text{mod$-A^e$}}((A^e)^{\vee}, A^{\vee}) \cong {\rm RHom}_{A^e}(A, A^e) % \cong{\rm RHom}_{A^e}(s^{2k} A^e, s^k A) \cong \Hom_{A^e}( s^{2k} A^e, s^k A)  \congs^{-k} A
\end{align}
where the first (quasi-)isomorphism follows since $s^{2k} A^e$ is dg projective and we have a natural isomorphism $s^{-k} A \cong{\rm Hom}_{\text{mod$-A^e$}}(s^{2k} A^e, s^k A)$ of left dg $A^e$-modules, the second one follows since $s^{2k}A^e \cong(A^e)^{\vee}$ and $s^k A \cong A^{\vee}$ as right dg $A^e$-modules, and the last one follows since the functor 
$$(-)^{\vee} \colon \text{$A^e$-mod}  \to \text{mod$-A^e$}, \quad X \mapsto X^{\vee}$$ is an equivalence between left and right finite-dimensional dg $A^e$-modules (since $(X^{\vee})^{\vee} \cong X$).
%It follows that $A^e$ is a {\it dg injective} $A$-$A$-bimodule, i.e. $ {\rm RHom}_{A^e}(X, A^e)\cong 0$ for any acyclic dg $A$-$A$-bimodule $X$.  This implies that the natural inclusion 
%\[
%\Hom_{A^e}^*(A, A^e) \hookrightarrow {\rm RHom}_{A^e}(A, A^e)\]
%is a quasi-isomorphism, where $\Hom_{A^e}^*(A, A^e)$ denotes the set of graded $A$-$A$-bimodule morphisms from $A$ to $A^e$. By Lemma \ref{lemma:bascipropertyeifi} we have the following isomorphism of dg $A$-$A$-bimodules
%\[
%s^{-k}A \xrightarrow{\cong} \Hom_{A^e}^*(A, A^e), \quad s^{-k}a \mapsto (b \mapsto \sum_i (-1)^{k|a|+(|a|-k)|e_i|} e_i a \otimes f_ib)
%\]
%with inverse $\varphi \mapsto  (-1)^k \sum_j \langle v_j, 1\rangle s^{-k}u_j,$ where we write $\varphi(1) := \sum_j u_j \otimes v_j$. Then the desired quasi-isomorphism $\varrho_A$ is the composition of the above two maps. 

On the other hand, by resolving $A$ as an $A$-$A$-bimodule using the bar resolution,  $A^!$ can be modeled as the Hochschild cochain complex $C^*(A, A^e)$. Recall that the $A$-$A$-bimodule structure on $C^*(A, A^e)$ is induced by the \lq inner action' of $A^e$ on $A^e$:
\begin{align}\label{inner-action}
(a \otimes b)\cdot (m \otimes n) = (-1)^{|a|(|b|+|m|)+|b||m|} mb \otimes an.
\end{align}
 Then the  quasi-isomorphism \eqref{align:leftrightmodules} can be lifted to the following map of dg $A$-$A$-bimodules
\begin{align}\label{align:CY}
\varrho_A \colon s^{-k}A \to C^*(A, A^e)
\end{align}
which sends $s^{-k}a$ to $\sum_i (-1)^{k|a|+(|a|-k)|e_i|} e_i a \otimes f_i \in C^0(A, A^e) = A^e$. Here,  we stress that the $A$-$A$-bimodule structure on $C^0(A, A^e) = A^e$ is given by the inner action \eqref{inner-action}. 
\end{proof}

\subsection{Hochschild chains and cochains}
\label{subsection2.2} Let $(A=\mathbb{K}.1 \oplus \overline{A}, d, \mu)$ be an augmented dg algebra where $\overline{A}$ is the kernel of the augmentation map. Note there is an isomorphism $\overline{A} \cong A / \mathbb{K}.1$. 

Recall that for any dg vector space $(V, d)$ we denote by $(s^iV, s^id)$ the $i$-th shifted space given by $(s^iV)^j = V^{i+j}$  and $s^id(v)=(-1)^i d (s^iv)$ for any $v\in V$. For simplicity, we write $\overline{a} $ for the element $sa\in s\overline{A} $ where $a \in \overline A$.

\begin{definition} \label{definition:hochschildchain}
Denote by $C_{-m,n}(A,A) = ((s\overline{A})^{\otimes m} \otimes A)^{n}$, i.e.\ elements in $(s\overline{A})^{\otimes m} \otimes A$ of total degree $n$. Let $C_n(A,A) = \bigoplus_{m \in \mathbb{Z}_{\geq 0}} C_{-m,n}(A,A)$ and $C_*(A,A)= \bigoplus_{n \in \mathbb{Z}}C_n(A,A)$. The \textit{Hochschild chain complex of $A$} is the complex $(C_*(A,A), \partial= \partial_v + \partial_h)$
where $\partial_v$ is the {\it internal} differential given by
\begin{equation*}
\begin{split}
\partial_v(\overline{ a_1}\otimes \cdots \otimes \overline{a_m}\otimes a_{m+1})=&-\sum_{i=1}^m(-1)^{\epsilon_{i-1}} \overline{a_1}\otimes \cdots \otimes \overline{a_{ i-1}} \otimes
\overline{d(a_i)} \otimes \overline{a_{i+1}}\otimes \cdots \otimes a_{m+1} \\
& +(-1)^{\epsilon_m} \overline{a_1}\otimes\cdots \otimes \overline{a_m} \otimes d(a_{m+1})
\end{split}
\end{equation*}
 and $\partial_h$ is the {\it external} differential given by
\begin{equation*}
\begin{split}
\partial_h(\overline{a_{1}}\otimes \cdots \otimes \overline{a_m}\otimes a_{m+1})=&\sum_{i=1}^{m-1} (-1)^{\epsilon_i} \overline{ a_1} \otimes \cdots \otimes \overline{a_{i-1}}\otimes \overline{a_{i}a_{i+1}}\otimes \overline{a_{i+2}} \otimes\cdots \otimes 
a_{m+1}\\
&-(-1)^{\epsilon_{m-1}} \overline{a_1}\otimes \cdots \otimes \overline{a_{m-1}}\otimes a_ma_{m+1}\\
&+ (-1)^{(|a_2|+
\dotsb+ |a_{m+1}| - m +1)|a_1| } \overline{a_2}\otimes \cdots \otimes \overline{a_m}\otimes a_{m+1}a_1.
\end{split}
\end{equation*}
Here we denote $\epsilon_i= |a_1| + \dotsb + |a_i| - i$ and  $\epsilon_0= 0$.
\end{definition}
\begin{remark}
Note that an element $
\overline{a_1}\otimes \cdots \otimes \overline{a_m}\otimes a_{m+1} \in (\sA)^{\otimes m} \otimes A$ belongs to $C_n(A, A)$ if and only if $|a_1| + |a_2| + \dotsb + |a_{m+1}| - m = n$. The differential $\partial$ on $C_*(A,A)$ is of degree $+1$. We use lower index notation in $C_*(A,A)$ to distinguish from the Hochschild cochain complex defined below.
 %In particular, if $A$ is concentrated in degree zero then  $C_*(A, A)$ is concentrated in non-positive degrees. 
\end{remark}

\begin{definition}
\label{definition-cohomology}
Denote by $C^{m,n}(A,A)= \text{Hom}^n_{\mathbb{K}}( (s\overline{A})^{\otimes m}, A)$, i.e.\ $\mathbb{K}$-linear maps $(s\overline{A})^{\otimes m} \to A$ of degree $n \in \mathbb Z$. Let $C^n(A,A) = \prod_{m \in \mathbb{Z}_{\geq 0}}  C^{m, n}(A, A)$ and $C^*(A,A) = \bigoplus_{n \in \mathbb{Z}} C^n(A,A)$. 
The \textit{Hochschild cochain complex of $A$} is the complex $(C^*(A,A), \delta= \delta^v + \delta^h)$ where $\delta^v$ is the {\it internal} differential given by
\begin{equation*}
\begin{split}
\delta^v(f)(\overline{a_{1}}\otimes \cdots\otimes \overline{a_m})={} &d(f(\overline{a_1}\otimes \cdots \otimes \overline{a_ m})) +\sum_{i=1}^m (-1)^{|f|+\epsilon_{i-1}} f(\overline{a_{1}}\otimes \cdots \otimes \overline{d(a_i)} \otimes  \cdots \otimes \overline{a_m}), 
\end{split}
\end{equation*}
and  $\delta^h$ is the {\it external} differential given by
\begin{equation*}
\begin{split}
\delta^h(f)(\overline{a_{1}} \otimes \cdots \otimes \overline{a_{m+1}})=& -(-1)^{(|a_1|-1)|f|} a_1 f(\overline{a_{2}}\otimes \cdots \otimes\overline{a_{ m+1}})  \\
&- \sum_{i=1}^m (-1)^{|f|+\epsilon_i}f(\overline{a_{1}}\otimes \cdots \otimes \overline{a_{ i-1}}\otimes\overline{ a_ia_{i+1} }\otimes \overline{a_{i}}\otimes \cdots \otimes \overline{a_{ m+1}}) \\
&+(-1)^{|f|+\epsilon_{m}} f(\overline{a_{1}}\otimes \cdots \otimes \overline{a_ m})a_{m+1},
\end{split}
\end{equation*}
where $\epsilon_i =  |a_1| + \dotsb + |a_i| - i$ and $\epsilon_0 =0$.
\end{definition}
Using similar formulas as the ones above, we may define for any dg $A$-$A$-bimodule $M$ the Hochschild chain and cochain complexes $C_*(A,M)$ and $C^*(A,M)$ by setting respectively  $$C_{n}(A,M)= \bigoplus_{m \geq 0} ((s \overline{A} )^{\otimes m} \otimes M)^n \quad \text{and} \quad C^n(A,M)= \prod_{m\geq 0} \text{Hom}^n_{\mathbb{K}}( (s\overline{A})^{\otimes m},M).$$  We denote $\HH^i(A,M) = \mathrm H^{i}( C^*(A,M))$ and $\HH_i(A,M)= \mathrm H^{-i}(C_*(A,M))$ for any $i \in \mathbb Z$.

\subsection{Goresky-Hingston algebra on reduced Hochschild homology}
Let $A$ be a \textit{connected} dg Frobenius algebra of degree $k>0$. In particular $A^0 \cong \mathbb{K} \cong A^k$. We follow the notation of Subection \ref{section2.1}.
\begin{definition} 
\label{definition:GH}
Define a product $\star\colon  C_*(A,A) \otimes C_*(A,A) \to C_*(A,A)$ of degree $k-1$ as follows: for any $\alpha = \overline{a_1} \otimes \dotsb \otimes \overline{a_p} \otimes a_{p+1}$ and $\beta = \overline{b_1} \otimes \dotsb \otimes \overline{b_q} \otimes b_{q+1}$ let
$$\alpha \star \beta  =\sum_i (-1)^{\eta_i}\overline{b_1}\otimes \cdots \otimes \overline{b_{q+1}e_i}\otimes \overline{a_1}\otimes \cdots \otimes \overline{a_p} \otimes a_{p+1}f_i,$$
where $\eta_i = |\alpha| |f_i| +|b_{q+1}| + (|\alpha|+k-1) (|\beta| +k-1)$. 
The product $\star$ induces a degree zero product on the $(1-k)$-shifted graded vector space $s^{1-k}C_*(A,A)$.
\end{definition}
Note that $\star$ does not satisfy the Leibniz rule with respect to the Hochschild chains differential $\partial$. In fact, %let $\alpha= \overline{a_1} \otimes \dotsb \otimes \overline{a_p} \otimes a_{p+1}$ and $\beta= \overline{b_1} \otimes \dotsb \otimes \overline{b_q} \otimes b_{q+1}$. 
if $p>0$ and $q>0$, we have (cf. Remark \ref{remark:iotaleibniz})
\begin{align}\label{leibniz-rule}
\partial( \alpha \star \beta) -\partial(\alpha) \star \beta - (-1)^{|\alpha|+k-1}\alpha \star \partial( \beta)=0,
\end{align}
but if $p=0$, so that $\alpha= a_1 \in C_{0, *}(A, A) =A$, then
$$\partial(\alpha \star \beta) -\partial(\alpha) \star \beta - (-1)^{|\alpha|+k-1}\alpha \star \partial(\beta)= \sum_i (-1)^{\eta_i + |\beta|-1-|b_{q+1}|} \overline{b_1} \otimes \dotsb \otimes \overline{b_q} \otimes b_{q+1}e_ia_1f_i.$$
A completely analogous computation yields that if $q=0$ there is a similar obstruction for $\star$ to satisfy the Leibniz rule. 
However, note that by degree reasons $e_ia_1f_i$ is only non-zero if $a_1 \in A^0 \cong \mathbb{K}$, and, in such case, $e_ia_1f_i \in A^k\cong \mathbb{K}$. Hence, $\star$ induces a non-unital  dg  algebra structure on the reduced Hochschild chain complex, defined as follows. 

\begin{definition} The \textit{reduced Hochschild chain complex}, denoted by $\overline{C}_*(A,A)$,  of a connected non-negatively graded dg algebra $A$ is the subcomplex of $C_*(A,A)$ defined as the \textit{complement} of $C_{0,0}(A,A)=A^0\cong \mathbb{K} \subset C_*(A,A)$. More precisely, $\overline{C}_*(A,A)$ is the total complex of the  sub-double complex $\overline{C}_{*,*}(A,A) \subset C_{*,*}(A,A)$ given by  $\overline{C}_{0,0}(A,A)=0$ and $\overline{C}_{i,j}(A,A) = C_{i,j}(A,A)$ for all pairs of integers $(i,j) \neq (0,0)$. 
\end{definition}
\begin{remark}
\label{remark-simply} Note that if $A$ is a simply connected dg algebra (i.e.\ $A^0 \cong \mathbb{K}$ and $A^1 =0$) then $$\overline{\HH}_i(A,A) = \begin{cases}
\HH_{i}(A,A) & \text{if $i>0$}\\
0 & \text{otherwise,}
\end{cases}
$$
since in this case we have $\overline C_{*\leq 0}(A, A) = 0$ and $\overline C_{*>0}(A, A) = C_{*>0}(A, A)$. In other words, we have $\HH_*(A, A)\cong \HH_0(A, A) \oplus \overline{\HH}_*(A, A) \cong \mathbb{K} \oplus \overline{\HH}_*(A, A)$.
\end{remark}

The following proposition is now easy to check. 

\begin{proposition} Let $A$ be a connected dg Frobenius algebra of degree $k$. Then the triple $(s^{1-k}\overline{C}_*(A,A), \partial, \star)$ is a (non-unital) dg algebra.
\end{proposition}
\begin{definition}
We call $(s^{1-k}\overline{\HH}_*(A,A), \star)$ the \textit{Goresky-Hingston algebra} on the reduced Hochschild homology of the connected dg Frobenius algebra $A$. 
\end{definition}

\begin{remark} The product $\star$ has been studied in the context of \textit{commutative} dg Frobenius algebras in \cite{Abb}, \cite{Kla}. In the commutative case, the complement of $C_{0,*}(A,A)=A\subset C_*(A,A)$ is a subcomplex of $C_*(A,A)$. The fact that $\star$ can be extended to the associative possibly non-commutative case has also been observed in \cite{Kau}.
\end{remark}

\begin{remark}\label{extension} If the Euler characteristic $\chi(A): = \mu\circ  \Delta(1) = \sum_i e_i f_i =0$ then the Leibniz rule \eqref{leibniz-rule} for $\star$ and the differential $\partial$ of the Hochschild chain complex always holds, namely, the failure of $\star$ in being a chain map vanishes in this case.  Hence, when $\chi(A)=0$ the Goresky-Hingston product is well-defined on the whole complex $\HH_*(A, A)$. 
\end{remark}

\section{Singular Hochschild cohomology algebra}
\label{section3}
We recall the definition of the singular Hochschild cochain complex and its cup product, as well as its invariance properties, following \cite{Wan} and \cite{RiWa}, that will be used in the proof of our main theorem. Throughout this section, we fix $A$ to be a unital dg algebra (no Frobenius structure is assumed).
\subsection{Singular Hochschild cochain complex} For any unital dg algebra $A$ and non-negative integer $p$ let $\Barr_{-p}(A):= A\otimes (\sA)^{\otimes p} \otimes A$. Note that $\Barr_{-p}(A)$ is a dg $A$-$A$-bimodule with differential 
\begin{align*}
d_{-p}(a_0 \otimes \overline{a_1} \otimes \dotsb \otimes \overline{a_p} \otimes a_{p+1}) ={} & d(a_0) \otimes \overline{a_1} \otimes \dotsb \otimes \overline{a_p} \otimes a_{p+1}  \\
&  - \sum_{i=1}^p (-1)^{|a_0| + \epsilon_{i-1}} a_0 \otimes \dotsb \otimes \overline{d(a_i)} \otimes \dotsb \otimes a_{p+1}  \\
& + (-1)^{|a_0| + \epsilon_p} a_0 \otimes \overline{a_1} \otimes \dotsb \otimes \overline{a_p} \otimes d(a_{p+1}),
\end{align*}
where $\epsilon_i = |a_1|+\dotsb+|a_i|-i.$  
For each positive integer $p$ denote by $$b_{-p}\colon \Barr_{-p}(A)\rightarrow \Barr_{-p+1}(A)$$ the map of degree one
\begin{equation*}
\begin{split}
b_{-p}(a_0 \otimes \overline{a_1} \otimes \dotsb \otimes \overline{a_p} \otimes a_{p+1}) ={} & (-1)^{
|a_0|}a_0a_1\otimes \overline{a_2}\otimes \cdots \otimes \overline{a_p}\otimes a_{p+1}\\
& + \sum_{i=1}^{p-1} (-1)^{|a_0| + \epsilon_i}a_0 \otimes \overline{a_1}\otimes \cdots \otimes \overline{a_ia_{i+1}}\otimes  \cdots \otimes \overline{a_p} \otimes a_{p+1}  \\
& - (-1)^{|a_0| +\epsilon_{p-1}}a_0\otimes \overline{a_1}\otimes \cdots \otimes \overline{a_{p-1}}\otimes a_pa_{p+1},
\end{split}
\end{equation*} 
where $\epsilon_i = |a_1|+\dotsb+|a_i|-i.$  

Set  $\Barr_*(A):= \bigoplus_{p=0}^{\infty} \Barr_{-p}(A)$. Note that $\Barr_*(A)$, equipped with the differential $b_* + d_*$ defined as above, is the normalized bar resolution of $A$ (cf. e.g. \cite[Section 1]{Abb1}). In particular,  
$$
b_{-p+1}\circ b_{-p} =0, \quad d_{-p}\circ d_{-p} =0, \quad b_{-p} \circ d_{-p} + d_{-p+1} \circ b_{-p} =0, \quad \quad \text{for any $p \geq 0$}.
$$ 
It follows from the above identity  $b_{-p} \circ d_{-p} + d_{-p+1}\circ b_{-p} =0$ that $b_{-p}$ is a morphism of dg $A$-$A$-bimodules of degree one.   %Define $b_{0}=s\mu: A \otimes A \to sA$ to be the multiplication of $A$. 
    %define $b: \Barr_*(A) \to \Barr_*(A)$ by $b:=b_0+b_{-1} + b_{-2} +\dotsb$, and let $\pi: A\twoheadrightarrow \sA$ be the natural projection map of degree one. Recall that for any graded vector space $V= \bigoplus_{j \in \mathbb{Z}} V^j$ we denote by $s^iV$ the shifted graded vector space given by $(s^iV)^j=V^{i+j}$ and denote by the same symbol $s^i: V \to s^iV$ the map of degree $-i$ that sends $x \in V^j$ to $s^{i}x \in (s^{i}V)^{j-i}$ for any $j \in  \mathbb Z$.  
%Note that $b_{-m}$ induces a map of degree $0$ $$b_{-m}s \colon s^{-1}\Barr_{-m}(A)\rightarrow \Barr_{-m+1}(A)$$ 
Define $\Omega_{\nc}^{p - 1}(A):= \text{Coker}(b_{-p})$ for $p \geq 1$. Since each $b_{-p}$ is a map of dg $A$-$A$-bimodules, $\Omega_{\nc}^{p-1}(A)$ inherits a dg $A$-$A$-bimodule structure.    %In particular, $\Omega_{\nc}^1(A)$ is the kernel of $s\mu s^{-1}: s(A \otimes A) \to sA$. 
Let $\pi \colon  A \twoheadrightarrow s\overline A$ be the natural projection of degree $-1$. Then we have the following description of $\Omega_{\nc}^p(A)$. 
\begin{lemma}\label{lemma-bimodule}
For each $p \in \mathbb{Z}_{\geq 0}$, there is a natural isomorphism of dg $A$-$A$-bimodules $$\alpha \colon \Omega_{\nc}^p(A)\xrightarrow{\cong} (\sA)^{\otimes p}\otimes A,$$ where the left $A$-module structure in $(\sA)^{\otimes p}\otimes A$ is given by $$a\blacktriangleright( \overline{a_1}\otimes\cdots \otimes \overline{a_p}\otimes a_{p+1}):=(\pi\otimes \id^{\otimes p})(b_{-p}(a\otimes \overline{a_1}\otimes\cdots \otimes \overline{a_p}\otimes a_{p+1})),$$ the right $A$-module structure is given by multiplication on the right $A$ factor of $(\sA)^{\otimes p}\otimes A$, and the differential on $(\sA)^{\otimes p}\otimes A$ is the tensor differential, i.e.\ given by

\begin{align*}
d(\overline{a_1}\otimes\cdots \otimes \overline{a_p}\otimes a_{p+1})=(-1)^{\epsilon_p}\overline{a_1}\otimes \dotsb\otimes \overline{a_p} \otimes d(a_{p+1}) -\sum_{i=1}^{p} (-1)^{\epsilon_{i-1}}\overline{a_1}\otimes \cdots \otimes d(\overline{a_i})\otimes \cdots \otimes a_{p+1},\end{align*} where $\epsilon_{i}=|a_1|+\dotsb+|a_{i-1}|+i-1$. 
\end{lemma}

\begin{proof}
Since $b_{-p} \circ b_{-p-1} =0$, the map $b_{-p}$ factors through  $\Omega_{\nc}^p(A)$. Namely, 
$$
\xymatrix@C=5pc{
& \Omega_{\nc}^p(A) \ar[rd]^-{\widetilde{b}_{-p}} & \\
\Barr_{-p-1} \ar[r]_-{b_{-p-1}}  & \Barr_{-p}(A) \ar[r]_-{b_{-p}} \ar@{->>}[u]^-{\widehat{b}_{-p}}& \Barr_{-p+1}(A),
}
$$   
where  $\widehat{b}_{-p}$ and $\widetilde{b}_{-p}$ are the natural  maps induced by the cokernel.

Define $\alpha \colon \Omega_{\nc}^p(A) \to (\sA)^{\otimes p}\otimes A$ as the composition $$\Omega_{\nc}^p(A)\xrightarrow{\widetilde{b}_{-p}} A\otimes (\sA)^{\otimes p-1}\otimes A \xrightarrow{\pi\otimes \id_{s\overline A}^{\otimes p-1} \otimes \id_A } \sA\otimes (\sA)^{\otimes p-1}\otimes A = (\sA)^{\otimes p} \otimes A. $$ It is clear that $\alpha$ is a morphism of dg $A$-$A$-bimodules. The inverse of $\alpha$ is given by the composition $\beta\colon (\sA)^{\otimes p}\otimes A \rightarrow A\otimes (\sA)^{\otimes p}\otimes A \xrightarrow{\widehat b_{-p}} \Omega_{\nc}^p(A),$ where the first map is $$\overline{a_1}\otimes \cdots \otimes \overline{a_p}\otimes a_{p+1} \mapsto 1\otimes \overline{a_1}\otimes \cdots \otimes \overline{a_p}\otimes a_{p+1}.$$ Note that we have $\alpha\circ \beta =\id$ and $\beta\circ \alpha =\id$. %$\alpha$ and $\beta$ clearly preserve the indicated $A$-$A$-bimodules structures. 
\end{proof}

\begin{remark}
\label{differentialform}
We have an exact sequence of dg $A$-$A$-bimodules (of degree one)
$$
\dotsb \xrightarrow{b_{-p-1}}\Barr_{-p}(A) \xrightarrow{b_{-p}} \Barr_{-p+1}(A)\xrightarrow{b_{-p+1}} \dotsb \xrightarrow{b_{-1}} \Barr_{0}(A) \xrightarrow{\mu} A \to 0
$$ 
since there is  a homotopy 
$
\chi_{-p} \colon \Barr_{-p+1}(A) \to \Barr_{-p}(A) 
$
given by 
$$
\chi_{-p} (a_0 \otimes \overline{a_1} \otimes \dotsb \otimes \overline{a_{p-1}} \otimes a_p) =  1 \otimes \overline{a_0} \otimes \overline{a_1} \otimes \dotsb \overline{a_{p-1}}\otimes a_p
$$
such that $\chi_{-p} \circ b_{-p} + b_{-p-1}\circ \chi_{-p-1} = \id_{\Barr_{-p}(A)}$, where for convenience we set $\Barr_{1}(A) =A$ and $b_{0} = \mu$. 
 This induces a short exact sequence of degree zero
\begin{align}
\label{shortexactsequence}
0\rightarrow s^{-1}\Omega_{\nc}^{p+1}(A)\hookrightarrow \Barr_{-p}(A) \rightarrow \Omega_{\nc}^{p}(A)\rightarrow 0
\end{align}
for any $p \geq 0$.
In particular, we have a short exact sequence of dg $A$-$A$-bimodules
$$ 
 0\rightarrow s^{-1}\Omega_{\nc}^{1}(A)\rightarrow A \otimes A \xrightarrow{\mu} A \to 0.
$$ 

For all $p, q\geq 0$ there is a natural isomorphism of dg $A$-$A$-bimodules 
$$
\kappa_{p, q}\colon \Omega^p_{\nc}(A) \otimes_A \Omega^q_{\nc}(A) \xrightarrow{\cong} \Omega^{p+q}_{\nc}(A)
$$ 
which sends $(\overline{a_1}\otimes \cdots \otimes \overline{a_p}\otimes a_{p+1}) \otimes_A  (\overline{b_1}\otimes \cdots \otimes \overline{b_p}\otimes b_{q+1})$ to 
$$
\overline{a_1}\otimes \cdots \otimes \overline{a_p}\otimes (a_{p+1}   \blacktriangleright(\overline{b_1}\otimes \cdots \otimes \overline{b_p}\otimes b_{q+1})).
$$
In particular, we get  
$ \Omega^p_{\nc}(A) \cong \underbrace{\Omega^1_{\nc}(A)\otimes_A  \dotsb \otimes_A \Omega^1_{\nc}(A)}_{p}$ for $p>0$.
For this reason, we call $\Omega_{\nc}^{p}(A)$ the {\it noncommutative differential $p$-forms of $A$}.
\end{remark}

\bigskip 

We identify $\Omega_{\nc}^p(A)$ with $(\sA)^{\otimes p}\otimes A$ via the isomorphism $\alpha$ in Lemma \ref{lemma-bimodule} from now on. Consider the Hochschild cochain complex $C^*(A, \Omega_{\nc}^p(A))$ with coefficients in the dg $A$-$A$-bimodule $\Omega_{\nc}^p(A)$. Namely, we have   
\begin{align*}
C^{m,n}(A,\Omega_{\nc}^p(A)) & = \text{Hom}_{\mathbb{K}}( (s\overline{A})^{\otimes m}, (\sA)^{\otimes p}\otimes A)^n\\ 
C^n(A, \Omega_{\nc}^p(A)) & = \prod_{m \in \mathbb{Z}_{\geq 0}} C^{m,n}(A,\Omega_{\nc}^p(A)).
\end{align*}

Define a morphism of (total) degree zero $${\theta}_{m, p}\colon C^{m, *}(A, \Omega_{\nc}^p(A))\rightarrow C^{m+1, *}(A, \Omega_{\nc}^{p+1}(A))$$ which sends $f \in \Hom_{\mathbb K}((\sA)^{\otimes m}, (\sA)^{\otimes p}\otimes A) $ to  $\theta_{m, p}(f) \in\Hom_{\mathbb K}((\sA)^{\otimes m+1}, (\sA)^{\otimes p+1}\otimes A) $ given by the following formula
\begin{equation}\label{equation-definition-theta}
 {\theta}_{m, p}(f)(\overline{a_1}\otimes \cdots\otimes \overline{a_{m+1}})=(-1)^{(|a_1|-1) |f| } \overline{a_1} \otimes f(\overline{a_2}\otimes \cdots \otimes \overline{a_{m+1}}).
\end{equation} 
Write $\theta_p = \prod_{m \in \mathbb Z_{\geq 0}} \theta_{m, p}$.  We have that $\theta_p$ is compatible with the differentials. Namely, the following diagram commutes
$$
\xymatrix{
C^{*}(A, \Omega_{\nc}^p(A)) \ar[r]^-{{\theta}_{p}} \ar[d]^{\delta}& C^{*}(A, \Omega_{\nc}^{p+1}(A))\ar[d]^-{\delta}\\
C^{*+1}(A, \Omega_{\nc}^p(A))\ar[r]^-{{\theta}_p} & C^{*+1}(A, \Omega_{\nc}^{p+1}(A)).
}
$$
The commutativity of the above diagram can also be verified immediately from formula \eqref{deRham} below.
Hence, the maps $\theta_p$ for $p\geq 0$  form an inductive system of cochain complexes.
\begin{definition}
Define
$$
\calC_{\sg}^{*}(A, A):= \varinjlim_{\theta_p} C^{*}(A, \Omega_{\nc}^{p}(A)), 
$$ as the colimit of the inductive system of cochain complexes
$$
C^*(A, A) \xrightarrow{\theta_0} C^*(A, \Omega_{\nc}^1(A))\xrightarrow{\theta_1}  \cdots \xrightarrow{\theta_{p-1}} C^*(A, \Omega_{\nc}^p(A)) \xrightarrow{\theta_p} C^{*}(A, \Omega_{\nc}^{p+1}(A)) \xrightarrow{\theta_{p+1}} \cdots.
$$
Since the $\theta_p$ are compatible with the Hochschild differentials, there is an induced differential $$\delta_{\sg} \colon \calC^{*}_{\sg}(A, A)\rightarrow \calC_{\sg}^{*+1}(A, A).$$ We call  $(\calC_{\sg}^*(A, A), \delta_{\sg})$ the {\it singular Hochschild cochain complex} of $A$. We denote its cohomology by $\HH_{\sg}^*(A, A)$ and call it {\it singular Hochschild cohomology} of $A$ (also called Tate-Hochschild cohomology of $A$).
\end{definition}

\subsection{Cup product on $\calC_{\sg}^*(A, A)$} 
\label{subsection:cupproduct}

Recall that  $C^{m, *}(A,\Omega_{\nc}^p(A))= \text{Hom}_{\mathbb{K}}( (s\overline{A})^{\otimes m}, (\sA)^{\otimes p} \otimes A).$

\begin{definition}
\label{defn-cup} Let $f\in C^{m, *}(A, \Omega_{\nc}^{p}(A))$ and $g\in C^{n, *}(A,\Omega_{\nc}^q(A))$. Define the cup product $f\cup g \in C^{m+n, *}(A, \Omega_{\nc}^{p+q}(A))$ by
$$
f\cup g:=(\id^{\otimes p+q}\otimes \mu)\circ (\id^{\otimes q}\otimes f\otimes \id) \circ  (\id^{\otimes m}\otimes g),
$$
where we have identified $\Omega_{\nc}^p(A)$ with $(s\overline{A})^{\otimes p}\otimes A$ as in Lemma \ref{lemma-bimodule} and denoted by $\mu\colon A \otimes A \to A$ the multiplication of $A$. More precisely,  $f\cup g$ is given by the following composition
\[
(s\overline A)^{\otimes m+n}\xrightarrow{\id^{\otimes m} \otimes g} (s\overline A)^{\otimes m+q}\otimes A \xrightarrow{\id^{\otimes q} \otimes f \otimes \id_A} (s\overline A)^{\otimes p+q} \otimes A \otimes A \xrightarrow{ \id^{\otimes p+q} \otimes \mu} (s\overline A)^{\otimes p+q} \otimes A.
\]
\end{definition}

When applying the above composition to an element the Koszul sign rule is used. For instance, write $g(\overline{a_{m+1}} \otimes \dotsb \otimes \overline{a_{m+n}}) = \sum_i \overline{b_{1,i}} \otimes \overline{b_{2, i}} \otimes \dotsb \otimes \overline{b_{q,i}} \otimes b_{q+1, i}$. If $m \geq q$ we have 
\[
f\cup g(\overline{a_1} \otimes \dotsb \otimes \overline{a_{m+n}})= \sum_i (-1)^\epsilon \overline{a_1} \otimes \dotsb \otimes \overline{a_q} \otimes f (\overline{a_{q+1}} \otimes \dotsb \otimes \overline{a_{m}} \otimes \overline{b_{1,i}} \otimes  \dotsb \otimes \overline{b_{q,i}}) b_{q+1, i}
\]
where $\epsilon = (|a_1|+\dotsb+|a_m|-m)|g| + (|a_1|+\dotsb+|a_q|-q)|f|$.

When $p=q=0$ the cup product coincides with the classical cup product on the Hochschild cochain complex $C^*(A, A)$. When $m=n=0$,  write $\alpha = \overline{a_1} \otimes \overline{a_2} \otimes \dotsb \otimes \overline{a_p} \otimes a_{p+1} \in \Omega^p_{\nc}(A) \cong \Hom_{\mathbb K} (\mathbb K, \Omega^p_{\nc}(A))$ and $\beta = \overline{b_1} \otimes \overline{b_2} \otimes \dotsb \otimes \overline{b_q} \otimes b_{q+1} \in \Omega^q_{\nc}(A) \cong \Hom_{\mathbb K} (\mathbb K, \Omega^q_{\nc}(A))$. Then $\alpha \cup \beta \in \Omega^{p+q}_{\nc}(A) \cong\Hom_{\mathbb K} (\mathbb K, \Omega^{p+q}_{\nc}(A))$ is given by 
\begin{align}
\label{cup}
\alpha \cup \beta = (-1)^{|\alpha|(|b_1|+\dotsb+|b_q|-q)} \overline{b_1} \otimes \dotsb \otimes \overline{b_q} \otimes \overline{a_1} \otimes \dotsb \otimes \overline{a_p} \otimes a_{p+1} b_{q+1}. 
\end{align}
Here the isomorphism $\Omega^p_{\nc}(A) \cong \Hom_{\mathbb K} (\mathbb K, \Omega^p_{\nc}(A))$ is given by $\alpha \mapsto (1 \mapsto \alpha)$. The above identity will be used in the proof of Proposition \ref{prop:iotastar}. 

Consider the {\it de Rham cocycle} $d_{\rm dR}\colon \overline{a} \mapsto \overline{a} \otimes 1$ in $C^{1, *}(A, \Omega_{\nc}^1(A))$. Note that  $\delta(d_{\rm dR})=0$. So $d_{\rm dR}$ is indeed a cocycle. 
Each structure map $\theta_p\colon  C^*(A, \Omega_{\nc}^p(A)) \to C^*(A, \Omega_{\nc}^{p+1}(A))$ is given by the cup product with $d_{\rm dR}$, namely
\begin{align} \label{deRham}
{\theta}_p(f)=f \cup d_{\rm dR}, \quad \quad \text{for any $f \in C^*(A, \Omega_{\nc}^p(A)).$}
\end{align}
 Since $d_{\rm dR} \cup f= f \cup d_{\rm dR}$, it follows that $\cup$ is compatible with the structure maps $\theta_p$, so it induces a well-defined product $$\cup\colon \calC_{\sg}^*(A, A) \otimes \calC_{\sg}^*(A, A) \to \calC_{\sg}^*(A, A).$$

The following is straightforward to verify. We refer to \cite[Section 4.1]{Wan} for details. 

\begin{proposition}\label{proposition:dgsingular} The cup product defines a unital dg algebra structure $(\calC_{\sg}^*(A, A), \delta_{\sg}, \cup)$. This induces a graded commutative algebra structure $(\HH_{\sg}^*(A, A), \cup)$. 
\end{proposition}

\bigskip 
In the rest of this subsection, we will show that the cup product $\cup$ on $\HH_{\sg}^*(A, A)$ may be obtained from the classical Hochschild cup product construction.  
Consider the short exact sequence of dg $A$-$A$-bimodules (cf. (\ref{shortexactsequence}))
$$
0\rightarrow s^{-1}\Omega_{\nc}^{p+1}(A)\hookrightarrow \Barr_{-p}(A) \rightarrow \Omega_{\nc}^{p}(A)\rightarrow 0.
$$
Apply the derived functor $\HH^*(A, - )$ to the above to obtain a long exact sequence
\begin{align}
\label{longexactsequence}
\cdots\rightarrow   \HH^*(A, \Barr_{-p}(A)) \rightarrow  \HH^*(A, \Omega_{\nc}^p(A)) \xrightarrow{\vartheta_{*,p}}  \HH^{*+1}(A, s^{-1}\Omega_{\nc}^{p+1}(A))\rightarrow \cdots 
\end{align}
where $\vartheta_{*,p}$ denotes the connecting morphism. Note that there is a natural isomorphism $$\HH^{*+1}(A, s^{-1}\Omega_{\nc}^{p+1}(A))\cong \HH^{*}(A, \Omega_{\nc}^{p+1}(A)).$$ Consider the following inductive system of graded vector spaces
$$
\HH^*(A, A)\xrightarrow{\vartheta_{*, 0}} \dotsb \xrightarrow{\vartheta_{*, p-1}} \HH^{*}(A, \Omega_{\nc}^p(A)) \xrightarrow{\vartheta_{*, p}} \HH^{*}(A, \Omega_{\nc}^{p+1}(A)) \xrightarrow{\vartheta_{*, p+1}}  \cdots
$$
and denote its colimit  by $\displaystyle \varinjlim_{\vartheta_p}\HH^{*}(A, \Omega_{\nc}^{p}(A)).$
%where the colimit  is taken along $r\in \Z$ such that $p+r\geq 0$.
%We call it the {\it singular Hochschild cohomology group} of degree $m-p$ of $A$.

\bigskip 
For any $p, q \geq 0$ and $m, n \in \mathbb Z$, consider the product defined as the composition
$$\cup'\colon \HH^m(A, \Omega_{\nc}^p(A))\otimes \HH^n(A, \Omega_{\nc}^q(A))\rightarrow \HH^{m+n} (A, \Omega_{\nc}^p(A) \otimes_A \Omega_{\nc}^q(A)) \xrightarrow{\cong} \HH^{m+n}(A, \Omega_{\nc}^{p+q}(A)),$$ where the first map is given by the classical Hochschild cup product construction and the second is induced by the isomorphism $\kappa_{p, q}\colon \Omega_{\nc}^p(A) \otimes_A \Omega_{\nc}^q(A) \xrightarrow{\cong} \Omega_{\nc}^{p+q}(A)$ in Remark \ref{differentialform}. Note that $\cup'$ can be lifted to the cochain complex level
$$
\cup'\colon C^{m, *}(A, \Omega^{p}_{\nc}(A)) \otimes C^{n, *}(A, \Omega_{\nc}^q(A))\rightarrow C^{m+n, *}(A, \Omega_{\nc}^{p+q}(A))$$
by 
\begin{align}
\label{cup'}
f \cup' g (\overline{a_1} \otimes \cdots\otimes \overline{a_{m+n}}) = (-1)^{\epsilon} \kappa_{p, q}\big(f(\overline{a_1} \otimes \dotsb \otimes \overline{a_m}) \otimes_A g(\overline{a_{m+1}}\otimes \dotsb \otimes \overline{a_{m+n}})\big),
\end{align}
where $\epsilon = (|a_1|+\dotsb+|a_m|-m)|g|.$ In general, $\cup'$ is not compatible with the structure maps $\theta_p$ at the cochain complex level. However, by the functoriality of the classical Hochschild cup product, $\cup'$ is compatible with the connecting morphism $\vartheta_{m, p}$ at the cohomology level. This induces a well-defined product (still denoted by $\cup'$) on the colimit $ \displaystyle \varinjlim_{\vartheta_p} \HH^{m}(A, \Omega_{\nc}^{p}(A)).$

\begin{proposition}
\label{proposition3.5}
There is a natural isomorphism of graded algebras
$$(\HH_{\sg}^*(A, A) , \cup) \cong ( \varinjlim_{\vartheta_p}\HH^{*}(A, \Omega_{\nc}^{p}(A)), \cup').$$
\end{proposition}

\begin{proof} First we claim that $
\mathrm H^*(\theta_p)=\vartheta_{p}$ for any $p\in \Z_{\geq 0}$. Indeed,  since the bar resolution $\Barr_*(A)$ is a projective resolution of $A$ as a dg $A$-$A$-bimodule,  we have
\begin{equation*}
\begin{split}
\HH^{m}(A, \Omega_{\nc}^p(A))&\cong \Hom_{\DD(A\otimes A^{\op})}(A, s^{m}\Omega_{\nc}^p(A) )\\
&\cong \Hom_{\KK(A\otimes A^{\op})}(\Barr_*(A), s^{m} \Omega_{\nc}^p(A) ),
\end{split}
\end{equation*}
where $\DD(A\otimes A^{\op})$ is the derived category of dg $A$-$A$-bimodules and $\KK(A\otimes A^{\op})$ is the homotopy category of dg $A$-$A$-bimodules. 
For any dg $A$-$A$-bimodule morphism $f \colon \Barr_{-n}(A) \to s^{m} \Omega_{\nc}^p(A)$ of degree zero, we have a commutative diagram
$$
\xymatrix{
& \Barr_{-n-1}(A) \ar[r]^-{b_{-n-1}} \ar[ld]_-{\theta_{n+1, p+1}(f)}  & \Barr_{-n}(A) \ar[ld]_-{\id\otimes \overline{f}} \ar[d]^-{f} \\
 s^{m-1}((\sA)^{\otimes p+1}\otimes A)\ar[r] &  s^{m}\Barr_{-p}(A) \ar@{->>}[r] & 
s^m((\sA)^{\otimes p}\otimes A)
}
$$
where the maps in the bottom row are from the two middle maps in the short exact sequence (\ref{shortexactsequence}), the map $\id\otimes \overline{f}\colon \Barr_{-n}(A) \rightarrow s^{m} \Barr_{-p}(A)$ is defined by $$(\id\otimes \overline{f})(a_0\otimes \overline{a_1} \otimes \dotsb \otimes \overline{a_{n}} \otimes a_{n+1})= (-1)^{|a_0|m} s^m (a_0 \otimes s^{-m} f(1\otimes \overline{a_1} \otimes \cdots  \otimes \overline{a_{n}} \otimes a_{n+1})),$$ and $\theta_{n+1, p+1}(f)$ is given by (\ref{equation-definition-theta}). This shows that $\mathrm H^*(\theta_p)$ is precisely the connecting morphism $$\vartheta_p\colon\HH^{*}(A, \Omega_{\nc}^p(A))\rightarrow \HH^{*}(A, \Omega_{\nc}^{p+1}(A))$$ of the long exact sequence \eqref{longexactsequence}.

%But note if we replace $\vartheta_p(f)$ by $\mathrm H^m(\theta_p)(f)$ in the dotted morphism the diagram still commutes, so by uniqueness it follows that $\mathrm H^*(\theta_p)=\vartheta_p$. This proves the claim. 

Since filtered colimits commute with homology, there is an isomorphism of  vector spaces $$\HH_{\sg}^*(A, A) \cong \varinjlim_{\theta_p}\HH^{*}(A, \Omega_{\nc}^{p}(A)) \cong \varinjlim_{\vartheta_p}\HH^{*}(A, \Omega_{\nc}^{p}(A)).$$
The fact that this is an isomorphism of algebras follows from the observation that the two products $\cup, \cup'\colon C^*(A,\Omega_{\nc}^p(A)) \otimes C^*(A, \Omega_{\nc}^{q}(A)) \to C^*(A, \Omega_{\nc}^{p+q}(A))$ (see (\ref{cup'}) and Definition \ref{defn-cup}) agree up to chain homotopy. More precisely, let $f\in C^{m, *}(A, \Omega_{\nc}^p(A))$ and $g\in C^{n, *}(A, \Omega_{\nc}^q(A))$. The chain homotopy for $f\cup g - f\cup'g $ is given by  $$
g \bullet_{<0 } f := \sum^q_{i=1} g \bullet_{-i} f,
$$
where $g\bullet_i f$ is given in Definition \ref{definitionbracket} below. 
That is, we have 
\begin{align}
\label{alignhomotopy}
f\cup g - f\cup'g = \delta( g \bullet_{<0} f) -  \delta(g)\bullet_{<0} f -(-1)^{|g|-1} g \bullet_{<0} \delta(f).
\end{align} 
When $q =0$, we note that   $f \cup g = f\cup' g$ and $g\bullet_{<0} f=0$. For the general case, identity (\ref{alignhomotopy}) may be verified by a direct computation. 
\end{proof}

\begin{definition}\label{definitionbracket} 
Let $f\in C^{m, *}(A, \Omega_{\nc}^p(A))$ and $g\in C^{n, *}(A, \Omega_{\nc}^q(A))$. Define a bracket $\{f, g\}\in C^{m+n-1, *}(A, \Omega_{\nc}^{p+q}(A))$ as 
$$
\{f, g\}=f\bullet g-(-1)^{(|f|+1)(|g| +1)} g\bullet f
$$
where  $$f\bullet g=\sum_{i=1}^m f\bullet_i g - \sum_{i=1}^p f\bullet_{-i} g$$ and $f\bullet_i g $ is defined by
\begin{equation*}
f\bullet_i g:=
\begin{cases}
\big(\id^{\otimes q}\otimes f\big) \circ \big(\id^{i-1} \otimes( (\id^{\otimes q}\otimes \pi) \circ g )\otimes \id^{\otimes m-i}\big) & \emph{for $1\leq i\leq m$},\\
\big(\id^{\otimes p+i} \otimes ((\id^{\otimes q}\otimes \pi) \circ g) \otimes \id^{\otimes -i}\big)\circ \big(\id^{\otimes n-1}\otimes f\big)  & \emph{for $-p\leq i\leq -1$},
\end{cases}
\end{equation*}
where $\pi\colon A\twoheadrightarrow \sA$ is the natural projection map of degree $-1$, and we identify 
$\Omega_{\nc}^p(A)$ with $(s\overline{A})^{\otimes p}\otimes A$ as in Lemma \ref{lemma-bimodule}. In particular, when $p=q=0$ we recover the classical Gerstenhaber bracket on 
$C^*(A, A)$. It follows from a direct calculation that the bracket is compatible with the colimit construction, thus the bracket is well-defined on $\calC_{\sg}^*(A, A)$. 
\end{definition}
We will not use the following result in this article but we recall it for general interest.

\begin{theorem}\cite[Corollary 5.3]{Wan}\label{theorem-ger}
The cup product $\cup$ and the bracket $\{ \cdot ,\cdot \}$ defined above give the singular Hochschild cochain complex $\calC_{\sg}^*(A, A)$ the structure of a unital dg algebra and a dg Lie algebra of degree $-1$, respectively. Moreover, $\cup$ and $\{ \cdot ,\cdot \}$  induce a Gerstenhaber algebra structure on $\HH_{\sg}^*(A, A).$
\end{theorem}

\subsection{Invariance of singular Hochschild cohomology under quasi-isomorphisms}
Let $A$ and $B$ be two unital dg algebras. Let  $\varphi\colon  A\rightarrow B$ be a morphism of dg algebras. Any dg $B$-$B$-bimodule can be viewed as a dg $A$-$A$-bimodule via $\varphi$.  %Note that $\varphi$  induces a morphism of dg $A$-$A$-bimodules $\Omega_{\nc}^p(\varphi)\colon \Omega_{\nc}^p(A)\rightarrow \Omega_{\nc}^p(B)$ for $p\in \Z_{\geq 0}$ given by 
%$$
%\Omega_{\nc}^p(\varphi)(\overline{a_1}\otimes \cdots \otimes \overline{a_p}\otimes a_{p+1})\colon=
%\overline{\varphi(a_1)}\otimes \cdots \otimes \overline{\varphi(a_{p})}\otimes \varphi(a_{p+1}),
%$$
%where we use Lemma \ref{lemma-bimodule} to identity $\Omega_{\nc}^p(A)$ with $(\sA)^{\otimes p}\otimes A$. Moreover, if $\varphi$ is a quasi-isomorphism, then so is $\Omega_{\nc}^p(\varphi)$. 
Then  $\varphi$ induces a map of dg $A$-$A$-bimodules $\Omega_{\nc}^p(\varphi)\colon \Omega_{\nc}^p(A)\rightarrow \Omega_{\nc}^p(B)$ for $p\in \Z_{\geq 0}$ given by 
$$
\Omega_{\nc}^p(\varphi)(\overline{a_1}\otimes \cdots \otimes \overline{a_p}\otimes a_{p+1}):=
\overline{\varphi(a_1)}\otimes \cdots \otimes \overline{\varphi(a_{p})}\otimes \varphi(a_{p+1}),
$$
where we use Lemma \ref{lemma-bimodule} to identity $\Omega_{\nc}^p(A)$ with $(\sA)^{\otimes p}\otimes A$ and $\Omega_{\nc}^p(B)$ with $(s\overline{B})^{\otimes p}\otimes B$. Moreover, if $\varphi$ is a quasi-isomorphism, then so is $\Omega_{\nc}^p(\varphi)$. 
\\

Now let us construct the singular Hochschild cochain complex $\calC_{\sg}^*(A, B)$ with coefficients in $B$. Consider the Hochschild cochain complex $C^*(A, \Omega_{\nc}^p(B))$ with coefficients in the dg $A$-$A$-bimodule $\Omega_{\nc}^p(B)$. Similar to (\ref{equation-definition-theta}),  we define a morphism of cochain complexes $$\theta^{A,B}_p\colon C^*(A, \Omega_{\nc}^p(B))\rightarrow C^*(A, \Omega_{\nc}^{p+1}(B))$$ which sends $f \in \Hom_{\mathbb K}((\sA)^{\otimes m}, (s\overline B)^{\otimes p}\otimes B)$ to $\theta^{A,B}_p(f) \in\Hom_{\mathbb K}((\sA)^{\otimes m+1}, (s\overline B)^{\otimes p+1}\otimes B)$ by 
$$
\theta^{A,B}_p(f) (\overline{a_1}\otimes \cdots \otimes \overline{a_{m+1}} ) = (-1)^{(|a_1|-1)|f|} \overline{\varphi(a_1)}\otimes f(\overline{a_2}\otimes
\cdots \overline{a_{m+1}} ).
$$
The maps $\theta^{A,B}_p$ form an inductive system of cochain complexes and we may define the {\it singular Hochschild cochain complex of $A$ with coefficients in $B$} as 
$$
\calC_{\sg}^*(A, B):= \varinjlim_{\theta^{A,B}_p} C^{*}(A, \Omega_{\nc}^{p}(B))
$$
with the induced differential. We denote its cohomology by $\HH_{\sg}^*(A, B)$.
\\

Observe that, for each $p\in \Z_{\geq 0}$, there is a zig-zag of morphisms of cochain complexes  
\begin{equation}\label{equation-zig-zag}
\xymatrix@C=4pc{
C^*(A, \Omega_{\nc}^p(A)) \ar[r] ^-{C^*(A,\Omega_{\nc}^p(\varphi))}& C^*(A, \Omega_{\nc}^p(B)) & C^*(B, \Omega_{\nc}^p(B))\ar[l]_-{C^*(\varphi, \Omega_{\nc}^p(B))}.
}
\end{equation}
These zig-zags are compatible with the inductive systems, thus we obtain an induced zig-zag of morphisms  between singular Hochschild cochain complexes
\begin{align}\label{align:zig-zagsg}
\xymatrix@C=4pc{
\calC_{\sg}^*(A, A) \ar[r] ^-{\calC_{\sg}^*(A, \varphi)}& \calC_{\sg}^*(A, B) & \calC^*_{\sg}(B, B)\ar[l]_-{\calC_{\sg}^*(\varphi, B)}.
}
\end{align}

\begin{lemma}
Let $\varphi\colon A\rightarrow B$ be a quasi-isomorphism of dg algebras. Then $\calC_{\sg}^*(A, \varphi)$ and $\calC_{\sg}^*(\varphi, B)$ are both quasi-isomorphisms, namely, the zig-zag above is one of quasi-isomorphisms. 
\end{lemma}

\begin{proof}
Note that all three complexes in the zig-zag (\ref{equation-zig-zag}) have complete decreasing filtrations with  the associated quotients 
$$
\xymatrix{
\Hom_{\mathbb K}((\sA)^{\otimes m}, \Omega_{\nc}^p(A))\ar[r] & \Hom_{\mathbb K}((\sA)^{\otimes m}, \Omega_{\nc}^p(B)) & 
\Hom_{\mathbb K}((s\overline B)^{\otimes m}, \Omega_{\nc}^p(B))\ar[l]
}
$$
These associated quotients are quasi-isomorphic, thus by the usual spectral sequence argument we may conclude that the morphisms in the zig-zag in (\ref{equation-zig-zag}) are quasi-isomorphisms for each $p\in \Z_{\geq 0}$. It follows that $\calC_{\sg}^*(A, \varphi)$ and $\calC_{\sg}^*(\varphi, B)$ are quasi-isomorphisms.
\end{proof}

\begin{remark}
As in Proposition \ref{proposition3.5},  we have a natural  isomorphism of graded vector spaces 
\begin{align}
\label{iso-ab}
\HH_{\sg}^*(A, B)\cong \varinjlim_{\vartheta_p^{A, B}} \HH^{*}(A, \Omega_{\nc}^{p}(B)),
\end{align}
where the colimit on the right hand side is taken along the connecting morphisms 
$$
\vartheta_p^{A, B} \colon \HH^{*}(A, \Omega_{\nc}^p(B)) \to \HH^{*+1}(A, s^{-1}\Omega_{\nc}^{p+1}(B)) =  \HH^{*}(A, \Omega_{\nc}^{p+1}(B)).
$$
 
In general, there is no natural cup product on $\calC_{\sg}^*(A, B)$ as in Definition \ref{defn-cup}. However, we have a well-defined cup product $\cup'$ at the cohomology level induced by the composition
\begin{align*}
\HH^m(A, \Omega_{\nc}^p(B))\otimes \HH^n(A, \Omega_{\nc}^q(B)) \rightarrow \HH^{m+n} (A, \Omega_{\nc}^p(B) \otimes_A \Omega_{\nc}^q(B))%\xrightarrow{\cong}
%\\ 
% \HH^{m+n}(A, \Omega_{\nc}^p(B) \otimes_B \Omega_{\nc}^q(B) ) & 
\xrightarrow{}  \HH^{m+n}(A, \Omega_{\nc}^{p+q} (B) ) ,
\end{align*}
where the first map is given by the classical Hochschild cup product construction using the dg $A$-$A$-bimodule structure on $\Omega_{\nc}^i(B)$ for $i=p,q$ via $\varphi\colon A \to B$, and the second isomorphism is induced by the following natural composition $$
\Omega_{\nc}^p(B) \otimes_A \Omega_{\nc}^q(B) \rightarrow \Omega_{\nc}^p(B) \otimes_B \Omega_{\nc}^q(B)\xrightarrow{\kappa_{p, q}} \Omega_{\nc}^{p+q}(B),
$$
where $\kappa_{p, q}$ is defined in Remark \ref{differentialform}. This yields a product  
\begin{align}
\label{mixedcup}\cup' \colon \HH^m(A, \Omega_{\nc}^p(B))\otimes \HH^n(A, \Omega_{\nc}^q(B))\rightarrow \HH^{m+n}(A, \Omega_{\nc}^{p+q}(B)).
\end{align}
By the  functoriality of the classical Hochschild cup product, we have $$
\vartheta_p^{A, B}(f) \cup' g =  f \cup' \vartheta_q^{A, B}(g) = \vartheta_{p+q}^{A, B}(f \cup' g)
$$ for any $f \in \HH^{m}(A, \Omega_{\nc}^p(B))$ and $ g \in \HH^{n}(A, \Omega_{\nc}^q(B))$, so the cup product $\cup'$ in (\ref{mixedcup}) induces a well-defined product on the colimit. Thus we obtain a product (via the isomorphism (\ref{iso-ab}))
$$
\cup'\colon \HH_{\sg}^*(A, B)\otimes  \HH_{\sg}^*(A, B) \to\HH_{\sg}^*(A, B) .
$$ 
\end{remark}

\begin{proposition}
\label{proposition-7.4}
The zig-zag of quasi-isomorphisms $$
\xymatrix@C=4pc{
\calC_{\sg}^*(A, A) \ar[r] ^-{\calC_{\sg}^*(A, \varphi)}& \calC_{\sg}^*(A, B) & \calC^*_{\sg}(B, B)\ar[l]_-{\calC_{\sg}^*(\varphi, B)}.
}
$$
induces isomorphisms of graded algebras
$$
\xymatrix@C=5.1pc{
(\HH_{\sg}^*(A, A), \cup) \ar[r]^-{\HH_{\sg}^*(A, \varphi)}  & (\HH_{\sg}^*(A, B), \cup') \ar[r]^-{\HH_{\sg}^*(\varphi, B)^{-1}}  & (\HH_{\sg}^*(B, B), \cup).
}
$$
\end{proposition}
\begin{proof}
First for any $m,n\in \mathbb Z$ and $p, q \in \mathbb Z_{\geq 0}$ we have the following commutative diagram 
\begin{align*}
\xymatrix{
\mathrm{HH}^m(A, \Omega_{\nc}^p(A)) \otimes \mathrm{HH}^n(A, \Omega_{\nc}^q(A)) \ar[d]_-{\mathrm{HH}^m(A,\Omega_{\nc}^p(\varphi))\otimes \mathrm{HH}^n(A,\Omega_{\nc}^q(\varphi))} \ar[r]^-{\cup'}  & \mathrm{HH}^{m+n}(A, \Omega_{\nc}^{p+q}(A))\ar[d]^-{\mathrm{HH}^{m+n}(A,\Omega_{\nc}^{p+q}(\varphi))} \\
\mathrm{HH}^m(A, \Omega_{\nc}^p(B)) \otimes \mathrm{HH}^n(A, \Omega_{\nc}^q(B))\ar[r]^-{\cup'}  & \mathrm{HH}^{m+n}(A, \Omega_{\nc}^{p+q}(B))\\
\mathrm{HH}^m(B, \Omega_{\nc}^p(B)) \otimes \mathrm{HH}^n(B, \Omega_{\nc}^q(B))\ar[u]^-{\mathrm{HH}^m(\varphi, \Omega_{\nc}^p(B))\otimes \mathrm{HH}^n(\varphi, \Omega_{\nc}^q(B))} \ar[r]^-{\cup'}  & \mathrm{HH}^{m+n}(B, \Omega_{\nc}^{p+q}(B))\ar[u]_-{\mathrm{HH}^{m+n}(\varphi, \Omega_{\nc}^{p+q}(B))}.
}
\end{align*}
The above commutative diagram is compatible with the structure maps $\vartheta$. It follows from Proposition \ref{proposition3.5} that $\cup =\cup'$ on $\HH_{\sg}^*(A, A)$ and $\HH_{\sg}^*(B, B)$, respectively. 
\end{proof}

\section{A quasi-isomorphism relating the Goresky-Hingston product and the singular Hochschild cup product}
\label{section4}
 In Subsections \ref{section-3.4} and \ref{subsection:4.2}, we assume that $A$ is a dg Frobenius algebra of degree $k>0$ and we follow the notation of Section \ref{section2}. 
 
\subsection{A homotopy retract between the Tate-Hochschild complex and the singular Hochschild cochain complex} \label{section-3.4}

\begin{definition}\label{definition:TH}
Let $A$ be a dg Frobenius algebra of degree $k >0$.  The \textit{Tate-Hochschild complex }of $A$, denoted by $(\calD^*(A, A), \delta)$, is the {\it totalization} complex of the double complex $\calD^{*,*}(A,A)$ obtained by connecting the Hochschild chains and cochains as
$$
\calD^{*,*}(A,A) = \cdots \to s^{1-k}C_{-1,*}(A,A) \to s^{1-k}C_{0,*}(A,A) \xrightarrow{\gamma} C^{0,*}(A,A) \to C^{1,*}(A,A) \to \cdots
$$
%where $\gamma$ is the composition )%$$
%s^{1-k}C_{0,*}(A,A) = s^{1-k}A \xrightarrow{\Delta} s(A \otimes A) \xrightarrow{sT} s(A \otimes A) \xrightarrow{s\mu} sA \xrightarrow{s^{-1}} A = C^{0,*}(A,A),
%$$ 
%and $T(x \otimes y)= (-1)^{|x||y|}y \otimes x$. 
where  the differential $\gamma$ is given by  $$
\gamma(a) = \sum_{i }(-1)^{|f_i||a|} e_i a f_i, \quad \text{for any $a \in A$,}$$  
where we recall that $\sum_{i } e_i\otimes f_i =\Delta(1) $ and $C^{0,*}(A,A) \cong A \cong C_{0, *}(A, A)$. 
By totalization we mean the direct sum totalization in the Hochschild chains direction and the direct product totalization in the Hochschild cochains direction. More precisely, for each $n\in \mathbb Z$ we have 
\begin{equation*}
\begin{split}
\calD^n(A, A)&= \prod_{p\in \Z_{\geq 0}} \Hom_{\mathbb{K}}((\sA)^{\otimes p}, A)^n \oplus \bigoplus_{p\in \Z_{\geq 0}}( (\sA)^{\otimes p}\otimes A)^{n-k+1}\\
&=C^n(A, A)\oplus C_{n-k+1}(A, A).
\end{split}
\end{equation*}
We denote by $\delta: \mathcal{D}^*(A,A) \to \mathcal{D}^{*+1}(A,A)$ the differential of the totalization. 
\end{definition}
In other words, the Tate-Hochschild complex $(\calD^*(A, A), \delta)$ is the mapping cone of the morphism 
\begin{align}\label{mappingcone}
\gamma \colon s^{-k} C_*(A, A) \to C^*(A, A)
\end{align}
where $\gamma(\alpha) = 0$ if $\alpha \in C_{-m, *}(A, A)$ for $m\neq 0$ and $\gamma(a) =  \sum_{i }(-1)^{|f_i||a|} e_i a f_i$ if $a \in A=C_{0, *}(A, A)$.

\begin{theorem}\label{thm:homtopyretract}
There exists a homotopy retract of cochain complexes.
\begin{align}
\label{equation-homotopy-transfer}
\begin{tikzpicture}[baseline=-2.6pt,description/.style={fill=white,inner sep=1pt,outer sep=0}]
\matrix (m) [matrix of math nodes, row sep=0em, text height=1.5ex, column sep=3em, text depth=0.25ex, ampersand replacement=\&, inner sep=1.5pt]
{
\calD^*(A, A) \& \calC_{\sg}^*(A, A) \\
};
\path[->,line width=.4pt,font=\scriptsize, transform canvas={yshift=.4ex}]
(m-1-1) edge node[above=-.4ex] {$\iota$} (m-1-2)
;
\path[->,line width=.4pt,font=\scriptsize, transform canvas={yshift=-.4ex}]
(m-1-2) edge node[below=-.4ex] {$\Pi$} (m-1-1)
;
\path[->,line width=.4pt,font=\scriptsize, looseness=8, in=30, out=330]
(m-1-2.355) edge node[right=-.4ex] {$H$} (m-1-2.5)
;
\end{tikzpicture}
\end{align}
Namely, $\iota$ and $\Pi$ are morphisms of cochain complexes such that $$ \Pi\circ \iota=\id\quad \text{and}\quad \id-\iota\circ \Pi =\delta_{\sg} \circ H  +H \circ \delta_{\sg}.$$
In particular, $\iota$ is a quasi-isomorphism of cochain complexes. 
\end{theorem}

Combining Theorem \ref{thm:homtopyretract} and the mapping cone construction \eqref{mappingcone}, we obtain a long exact sequence
\begin{align}\label{align:longexact}
\dotsb \to \HH_{i-k} (A, A) \to \HH^i(A, A) \to \HH^i_{\sg}(A, A) \to \HH_{i-k+1}(A, A)\to \dotsb.
\end{align}
where we identify $\HH_{\sg}^i(A, A)$ with ${\mathrm H}^i(\calD^*(A, A))$ via the natural  isomorphism $\mathrm{H}^i(\iota)$.

For the purposes of this paper, we only use the fact that the map $\iota$ is a quasi-isomorphism from the above theorem. We define $\iota$ in Definition \ref{definitioniota} below. However, for the sake of clarity and organization, we refer to the appendix at the end of the paper for the definition of $\Pi$ and $H$ and for the proof of Theorem \ref{thm:homtopyretract}. 

%\begin{remark} For the purposes of this paper we do not use the full strength of Theorem 4.2, only the fact that $\iota$ is a quasi-isomorphism, as discussed in Section 4.2. However, having a homotopy retract is useful in studying the full chain level algebraic structure on $\mathcal{D}^*(A,A)$ obtained by homotopy transfer methods. More precisely, the dg algebra structure of $\calC_{\sg}^{*}(A, A)$ may be transferred to an $A_{\infty}$-algebra structure on $\mathcal{D}^*(A,A)$ along the homotopy retract described in the theorem. Some aspects regarding the full algebraic structure are studied in \cite{RiWa}. We expect that this chain level algebraic structure on $\mathcal{D}^*(A,A)$ can be described geometrically in terms of string topology and leads to fine invariants associated to manifolds. However, in this paper we are only concerned with the Goresky-Hingston algebra, which may be realized as a subalgebra of the $A_{\infty}$-algebra structure on $\mathcal{D}^*(A,A)$, as we discuss next. 
%\end{remark}

\begin{definition}
\label{definitioniota} Define $$\iota\colon \calD^*(A, A)\hookrightarrow \calC_{\sg}^*(A, A)$$ as follows. 
%For any $n\in \Z$, define  $\iota_n\colon \calD^n(A, A) \rightarrow \calC_{\sg}^n(A, A)$ by
\begin{enumerate}
\item If $f\in C^n(A, A)$, define $\iota(f)=f \in C^n(A, A)\subset \calC_{\sg}^n(A, A)$.
\item If $\alpha=\overline{a_1}\otimes \cdots \otimes \overline{a_p}\otimes a_{p+1} \in C_{*}(A, A)$ define $$\iota(\alpha)\in \Hom_{\mathbb{K}}(\mathbb{K}, (\sA)^{\otimes p+1}\otimes A)\subset \calC^*_{\sg}(A, A)$$ by
$$\iota(\alpha)(1_{\mathbb K})=\sum_{i} (-1)^{ |f_i||\alpha| } \overline{e_i}\otimes \overline{a_1}\otimes \cdots \otimes \overline{a_p}\otimes a_{p+1}f_i. $$
\end{enumerate}
The map $\iota$ is clearly an injection. We refer to the appendix for the proof that $\iota $ is compatible with the differentials. 
\end{definition}

\subsection{Relating the Goresky-Hingston product and the singular Hochschild cup product}\label{subsection:4.2}

We now show that $\iota$ intertwines the products $\star$ and $\cup$. 

\begin{proposition}\label{prop:iotastar}
Let $\alpha= \overline{a_1} \otimes \dotsb \otimes \overline{a_p} \otimes a_{p+1}$ and $\beta= \overline{b_1} \otimes \dotsb \otimes \overline{b_q} \otimes b_{q+1}$ be two elements in $C_*(A,A)$. Then we have 
$$
\iota(\alpha) \cup \iota(\beta) = \iota(\alpha \star \beta).
$$
\end{proposition}
\begin{proof}
By Definition \ref{definition:GH} we have 
%Finally we show that the map $\iota$ sends the product $\star$ on Hochschild chains to the cup product $\cup$ on the singular Hochschild cochain complex. Let $\alpha= \overline{a_1} \otimes \dotsb \otimes \overline{a_p} \otimes a_{p+1}$ and $\beta= \overline{b_1} \otimes \dotsb \otimes \overline{b_q} \otimes b_{q+1}$ be two elements in $C_*(A,A)$. We prove that $\iota(\alpha\star \beta) = \iota(\alpha) \cup \iota(\beta)$. Note that $\iota(\alpha \star \beta) \in \text{Hom}_{\mathbb{K}}(\mathbb{K}, \Omega_{\nc}^{p+q+2}(A))$ is determined by
\begin{align*}
\iota(\alpha \star \beta)(1_{\mathbb K})={} & \sum_i (-1)^{\eta_i} \iota\Big(\overline{b_1}\otimes \cdots \otimes \overline{b_{q+1}e_i}\otimes \overline{a_1}\otimes \cdots \otimes \overline{a_p} \otimes a_{p+1}f_i\Big)(1_{\mathbb K}) 
\\
={} & \sum_{i,j} (-1)^{\eta_i+\eta_j'} \overline{e_j} \otimes  \overline{b_1}\otimes \cdots \otimes \overline{b_q} \otimes \overline{b_{q+1}e_i}\otimes \overline{a_1}\otimes \cdots \otimes \overline{a_p} \otimes a_{p+1}f_if_j,
\end{align*}
where  
$ \eta_i =  |\alpha| |f_i| +|b_{q+1}| + (|\alpha|+k-1) (|\beta| +k-1)$ and $
\eta'_j  = |f_j| (|\alpha|+|\beta|+k-1)$. 
Using the identities in Remark \ref{remark:usefulidentities}, the above term equals
$$\sum_{i,j} (-1)^{\eta_i+\eta'_j + |b_{q+1}|(k + |\alpha|)}\overline{e_j} \otimes \overline{b_1}\otimes \cdots \otimes \overline{b_q} \otimes \overline{e_i}\otimes \overline{a_1}\otimes \cdots \otimes \overline{a_p} \otimes a_{p+1}f_i b_{q+1} f_j,$$
which is precisely $(\iota(\alpha) \cup \iota(\beta))(1_{\mathbb K})$ (cf. (\ref{cup})). 
\end{proof}

\begin{remark}\label{remark:iotaleibniz}
Note that Proposition \ref{prop:iotastar} yields  \eqref{leibniz-rule} since $\iota$ is an injection of complexes. %\commentz{If $A$ is connected dg Frobenius algebra of degree $k$ then the Goresky-Hingston algebra $may be viewed as a dg subalgebra of $\calC_{\sg}^{*}(A, A)$.}
\end{remark}

\subsection{Another mapping cone model for singular Hochschild cohomology}\label{E(A,A)}
We now define a new cochain complex $\mathcal{E}^*(A,A)$ associated to any dg algebra $A$ without requiring any Frobenius structure. 
This new complex $\mathcal{E}^*(A,A)$ satisfies the following properties: 
\begin{enumerate}
\item $\mathcal{E}^*(A,A)$ arises as the mapping cone of a natural map $A \otimes^{\mathbb{L}}_{A^e} A^! \to \text{RHom}_{A^e}(A,A)$ associated to any dg algebra $A$.
\item When $A$ is equipped with a dg Frobenius structure we factor the quasi-isomorphism $\iota\colon \mathcal{D}^*(A,A) \to C^*_{\sg}(A,A)$ (Definition \ref{definitioniota}) as a composition of quasi-isomorphisms $$\mathcal{D}^*(A,A) \xrightarrow{\tau} \mathcal{E}^*(A,A) \xrightarrow{\widetilde{\iota}} \calC_{\sg}^*(A,A).$$ 
\item $\mathcal{E}^*(A,A)$ enjoys better functoriality  properties than $\mathcal{D}^*(A,A)$ and we will use these in the proof of our main theorem. 
\end{enumerate}

Recall that the inverse dualizing complex $A^!=\mathrm{RHom}_{A^e}(A, A^e)$ can be modeled as the Hochschild cochain complex $C^*(A, A^e)$ equipped with the dg $A$-$A$-bimodule structure given by the inner $A$-$A$-bimodule action of $A \otimes A$. From now on we will denote $A^!=C^*(A,A^e)$. By the functoriality of $C^*(A, -)$, there is a natural morphism of complexes 
\begin{align}
\label{align:mu}
C^*(A, \mu) \colon A^! \twoheadrightarrow C^*(A, A) 
\end{align}
induced by the multiplication $\mu$ of $A$

Consider the Hochschild chain complex $C_*(A, A^!)$ of $A$ with coefficients in $A^!$. The above morphism \eqref{align:mu} induces a morphism of complexes 
\begin{align}\label{widehatgamma}
\widetilde \gamma \colon C_*(A, A^!) \to C^*(A, A)
\end{align}
given by 
\[
\widetilde \gamma(\alpha) = \begin{cases}
C^*(A, \mu)(\alpha) & \text{if $\alpha \in C_{0, *}(A, A^!) = A^!$}\\
0 & \text{otherwise.}
\end{cases}
\]

\begin{definition}
Given any dg algebra $A$, define the cochain complex $\calE^*(A, A)$ as the mapping cone of the morphism $\widetilde  \gamma$. 
\end{definition}
Similar to \eqref{align:longexact}, since $\mathcal{E}^*(A,A)$ is defined as a mapping cone, we have a long exact sequence
\begin{align}\label{align:exact2}
\dotsb \to \HH_i(A, A^!) \to \HH^i(A, A) \to \mathrm H^i(\calE^*(A, A)) \to \HH_{i+1}(A, A^!) \to \HH^{i+1}(A, A) \to \dotsb
\end{align}

The complex $\calE^*(A, A)$ may be viewed as a subcomplex of $\calC_{\sg}^*(A, A)$ as follows. Define a morphism of complexes 
\begin{align}\label{widetildeiota}
\widetilde\iota \colon \mathcal E^*(A, A) \hookrightarrow \calC_{\sg}^*(A, A), %\quad \overline{a_1}\otimes \dotsb \otimes \overline{a_m} \otimes \varphi_n \mapsto (\overline{b_1}\otimes \dotsb \otimes \overline{b_n} \mapsto \sum_j (-1)^{\xi_j} \overline{u_j} \otimes \overline{a_1} \otimes \dotsb \otimes \overline{a_m} \otimes v_j)
\end{align}
 which sends $\varphi \in C^*(A, A)$ to $\varphi \in C^*(A, A) \subset \calC_{\sg}^*(A, A)$ and sends  $\overline{a_1}\otimes \dotsb \otimes \overline{a_m} \otimes \varphi_n \in C_{-m, *}(A, A^!)$, where $\varphi_n \in C^{n, *}(A, A\otimes A)$,   to an elment in $C^{n, *}(A, \Omega_{\nc}^{m+1}(A)) \subset \mathcal \calC_{\sg}^*(A, A) $ given by  $$\overline{b_1} \otimes \dotsb\otimes \overline{b_n} \mapsto   \sum_j (-1)^{\epsilon_j} \overline{u_j} \otimes \overline{a_1} \otimes \dotsb \otimes \overline{a_m} \otimes v_j \in \Omega_{\nc}^{m+1}(A),$$ where we write $\varphi_n(\overline{b_1} \otimes \dotsb\otimes \overline{b_n}) = \sum_j u_j \otimes v_j \in A \otimes A$ and $\epsilon_j = |v_j|(|a_1|+\dotsb+|a_m|-m)$. Similar to the map $\iota$ in Definition \ref{definitioniota},  $\widetilde \iota$ is  an injective map  of complexes.

Let  $A$ be a dg Frobenius algebra of degree $k$. Recall from \eqref{align:CY} the natural quasi-isomorphism
$ \varrho_A \colon s^{-k} A \xrightarrow{\simeq} A^!$ of dg $A$-$A$-bimodules. There is an induced commutative square
\[
\xymatrix{
C_*(A, s^{-k}A) \ar[r]^-{\gamma} \ar[d]_{C_*(A, \varrho_A)} & C^*(A, A)\ar[d]^-{=} \\
C_*(A, A^!) \ar[r]^-{\widetilde \gamma} & C^*(A, A),
}
\] 
where $C_*(A, \varrho_A)$ is a quasi-isomorphism and $\gamma$ is given in \eqref{mappingcone}. This yields  a quasi-isomorphism between the mapping cones 
\[
\tau \colon \calD^*(A, A) \hookrightarrow \calE^*(A, A).
\]

\begin{proposition}\label{factor}
Let $A$ be a dg Frobenius algebra of degree $k$. Then 
the quasi-isomorphism $\iota\colon \mathcal{D}^*(A,A) \to C^*_{\sg}(A,A)$ in Definition \ref{definitioniota} is the composition of the morphisms $$\mathcal{D}^*(A,A) \xrightarrow{\tau} \mathcal{E}^*(A,A) \xrightarrow{\widetilde{\iota}} \calC_{\sg}^*(A,A).$$ 
%\item $\mathcal{E}^*(A,A)$
%we have the following commutative diagram 

%\begin{align}
%\xymatrix{
%\mathcal D^*(A, A) \ar[rd]^-{\iota} \ar[d]^-{\tau}& \\
%\mathcal E^*(A, A) \ar[r]^-{\widetilde\iota} & \calC_{\sg}^*(A, A).\\
%}
%\end{align}
As a result, $\widetilde{\iota}$ is a quasi-isomorphism and there is a commutative diagram between the long exact sequences \eqref{align:longexact} and \eqref{align:exact2}
\[
\xymatrix@R=1.5pc{
\dotsb \ar[r]&  \HH_{i-k}(A, A) \ar[r] \ar[d]^{\cong}_{\HH_i(A, \varrho_A)}&  \HH^i(A, A) \ar[r] \ar[d]^{=}& \HH_{\sg}^i(A, A) \ar[r] \ar[d]^{{\rm H}^i(\widetilde \iota)^{-1}}_{\cong} &  \HH_{i-k+1 }(A, A) \ar[d]_-{\cong}^{\HH_{i+1}(A, \varrho_A)}\ar[r] & \dotsb\\
%&&& &&\HH_{\sg}^i(A, A)\\
\dotsb \ar[r]&  \HH_{i}(A, A^!) \ar[r] &  \HH^i(A, A) \ar[r] & \mathrm H^i(\calE^*(A, A)) \ar[r] & \HH_{i+1}(A, A) \ar[r] &  \dotsb
}
\]
 %so it induces an isomorphism $\emph{H}^*(\mathcal{E}^*(A,A)) \cong \HH_{\sg}^*(A,A)$. 
\end{proposition}
\begin{proof}The first part is clear. The cochain map $\widetilde{\iota}$ is a quasi-isomorphism since both $\iota$ and $\tau$ are quasi-isomorphisms. The commutativity of the diagram follows  since $\tau$ is a map between mapping cones. Note that in the diagram we identify $\HH_{\sg}^i(A, A)$ with $\mathrm{H}^i(\calD^*(A, A))$ via $\mathrm{H}^i(\iota)$. 
\end{proof}

We now discuss the functoriality properties of $\mathcal{E}^*(A,A)$ that will be used in the proof of our main theorem in Section \ref{section5}. 

Let $\varphi\colon A \to B$ be a morphism of dg algebras. Then $B^e$ may be viewed as an $A$-$A$-bimodule. We have a morphism of complexes induced by the multiplication $\mu_B$ of $B$
\[
C^*(A, \mu_B)\colon C^*(A, B^e) \to C^*(A, B).
\]
 We denote by $\calE^*(B, C^*(A, B^e))$ the mapping cone of the following natural map (cf. \eqref{widehatgamma})  $$ C_*(B, C^*(A, B^e)) \to C^*(A, B)$$ which sends $C_{0, *}(B, C^*(A, B^e)) = C^*(A, B^e)$ to $C^*(A, B)$ via the map $C^*(A, \mu_B)$ and sends $C_{i, *}(B, C^*(A, B^e))$ to zero for $i \neq 0$. Similar to \eqref{widetildeiota}, we have a morphism of complexes
\[
\widetilde \iota_{A, B} \colon \calE^*(B, C^*(A, B^e)) \to \calC_{\sg}^*(A, B).
\]

By functoriality,   $\varphi$ induces two morphisms 
\begin{equation}\label{algin:phiinducemap}
\begin{split}
\calE^*(A, A)  &\to \calE^*(B, C^*(A, B^e))\\
\calE^*(B, B)  & \to \calE^*(B, C^*(A, B^e)). 
\end{split}
\end{equation}

 \begin{proposition}\label{functoriality}
Any morphism $\varphi \colon A\to B$ of dg algebras induces a commutative diagram
\begin{align}\label{align:functorialityproperty}
\xymatrix{
\calE^*(A, A) \ar[r]\ar[d]_-{\widetilde\iota} & \calE^*(B, C^*(A, B^e)) \ar[d]^-{\widetilde \iota_{A, B}}& \calE^*(B, B) \ar[l]\ar[d]^-{\widetilde\iota} \\
\calC_{\sg}^*(A, A) \ar[r]_-{\calC_{\sg}^*(A, \varphi)}& 
 \calC_{\sg}^*(A, B)&  \calC_{\sg}^*(B, B)\ar[l]^-{\calC_{\sg}^*(\varphi, B)}
}
\end{align}
where the horizontal maps on the top are given in \eqref{algin:phiinducemap}. 
Furthermore, if $\varphi$ a quasi-isomorphism, then all the horizontal morphisms in \eqref{align:functorialityproperty} are quasi-isomorphisms.
\end{proposition}
Finally, we relate the long exact sequences associated to quasi-isomorphic dg Frobenius algebras.

\begin{corollary}\label{corollary:longexactsequence}
Let $(A, \langle -, -\rangle_A)$ and $(B, \langle -, -\rangle_B)$ be two dg Frobenius algebras of degree $k$. 
Suppose that there is a zig-zag of quasi-isomorphisms of dg algebras $$A \xleftarrow{\simeq} \bullet \xrightarrow{\simeq} \dotsb \xleftarrow{\simeq} \bullet \xrightarrow{\simeq} B.$$ Then there is a commutative diagram between long exact sequences 
\[
\xymatrix@R=1.5pc{
\dotsb \ar[r]&  \HH_{i-k}(A, A) \ar[r] \ar[d]^{\cong}_{\HH_i(A, \varrho_A)}&  \HH^i(A, A) \ar[r] \ar[d]^{=}& \HH_{\sg}^i(A, A)\ar[r] \ar[d]^{{\rm H}^i(\widetilde\tau)^{-1}}_{\cong} &  \HH_{i-k+1 }(A, A) \ar[d]_-{\cong}^{\HH_{i+1}(A, \varrho_A)}\ar[r] & \dotsb\\
%&&& &&\HH_{\sg}^i(A, A)\ar[ddd]^-{\cong}\\
\dotsb \ar[r]&  \HH_{i}(A, A^!) \ar[r] \ar[d]^{\cong}&  \HH^i(A, A) \ar[r] \ar[d]^{\cong}& \mathrm H^i(\calE^*(A, A)) \ar[r] \ar[d]^{\cong}& \HH_{i+1}(A, A^!)\ar[d]^-{\cong} \ar[r] &\dotsb\\
\dotsb \ar[r]&  \HH_{i}(B, B^!) \ar[r]&  \HH^i(B, B) \ar[r] & \mathrm H^i(\calE^*(B, B))  \ar[r]   &\HH_{i+1}(B, B^!) \ar[r] &\dotsb\\
%&&&&& \HH_{\sg}^i(B, B)\\
\dotsb \ar[r]&  \HH_{i-k}(B, B)\ar[u]_-{\cong}^-{\HH_i(B, \varrho_B)} \ar[r]&  \HH^i(B, B) \ar[r] \ar[u]^-{=} &\HH_{\sg}^i(B, B) \ar[u]_-{{\rm H}^i(\widetilde \tau)^{-1}}^-{\cong} \ar[r] &  \HH_{i-k+1 }(B, B) \ar[u]^-{\cong}_-{\HH_{i+1}(B, \varrho_B)}\ar[r] &\dotsb
}
\]
where the vertical isomorphisms in the middle row are induced by the zig-zag of maps in \eqref{algin:phiinducemap}, and the vertical isomorphisms in the top and bottom rows are given in Proposition \ref{factor}. 

Moreover, the isomorphism between $\HH_{\sg}^i(A, A)$ and $\HH_{\sg}^i(B, B)$ in the diagram coincides with the isomorphism induced by the zig-zag \eqref{align:zig-zagsg}. 
%\[
%\xymatrix@R=0.8pc{
%\dotsb \ar[r]&  \HH_{i-k}(A, A) \ar[r] \ar[dd]^{\cong}_{\HH_i(A, \varrho_A)}&  \HH^i(A, A) \ar[r] \ar[dd]^{=}& \mathrm H^i(\calD^*(A, A)) \ar[r] \ar[dd]^{{\rm H}^i(\tau)}_{\cong} \ar[rrd]^-{\cong}_-{{\mathrm H}^i(\iota)}&  \dotsb\\
%&&& &&\HH_{\sg}^i(A, A)\ar[ddd]^-{\cong}\\
%\dotsb \ar[r]&  \HH_{i}(A, A^!) \ar[r] \ar[d]^{\cong}&  \HH^i(A, A) \ar[r] \ar[d]^{\cong}& \mathrm H^i(\calE^*(A, A)) \ar[r] \ar[d]^{\cong}\ar[rru]_-{\cong}^-{{\mathrm H}^i(\widetilde\iota)}&  \dotsb\\
%\dotsb \ar[r]&  \HH_{i}(B, B^!) \ar[r]&  \HH^i(B, B) \ar[r] & \mathrm H^i(\calE^*(B, B))  \ar[r] \ar[drr]^-{\cong}_-{{\mathrm H}^i(\widetilde\iota)}    &\dotsb\\
%&&&&& \HH_{\sg}^i(B, B)\\
%\dotsb \ar[r]&  \HH_{i-k}(B, B)\ar[uu]^-{\cong}_{\HH_i(B, \varrho_B)} \ar[r]&  \HH^i(B, B) \ar[r] \ar[uu]^-{=} & \mathrm H^i(\calD^*(B, B))\ar[uu]^-{{\rm H}^i(\tau)}_{\cong} \ar[rru]_-{\cong}^-{{\mathrm H}^i(\iota)} \ar[r] &  \dotsb
%}
%\]
%where the vertical isomorphisms in the middle row are induced by the zig-zag of maps in \eqref{algin:phiinducemap}, the isomorphism $\HH_{\sg}^*(A, A) \cong \HH_{\sg}^*(B, B)$ is induced by the zig-zag between $A$ and $B$ as explained in Proposition \ref{proposition-7.4}, and the other isomorphisms are given as labeled in the diagram. 
\end{corollary}
\begin{proof}
The commutativity of the squares in the top and bottom rows follows from Proposition \ref{factor}.
This commutativity of the squares in the middle row directly follows from the fact that $\calE^*$ is defined as mapping cones.  The last statement follows from Proposition \ref{functoriality}. \end{proof}

%Let $\varphi \colon A\to B$ be a quasi-isomorphism of dg algebras. Assume that $s^{-k} A \simeq A^!$ as dg $A$-$A$-bimodules

\section{Proof of the main theorem}\label{section5}
We will now prove that, for simply connected dg Frobenius algebras,  the isomorphism class of the Goresky-Hingston algebra is invariant under quasi-isomorphisms. To conclude this, we use the invariance of the singular Hochschild cohomology algebra with respect to quasi-isomorphisms together with the quasi-isomorphism relating the Tate-Hochschild complex and the singular Hochschild cochain complex described in the previous section.

\subsection{Tate-Hochschild cohomology in the simply connected case}
For simply connected dg Frobenius algebras, Tate-Hochschild cohomology may be described in terms of Hochschild homology and cohomology as follows.
\begin{proposition}
\label{proposition4.7}
Let $A$ be a simply connected dg Frobenius algebra of degree $k$. Then 
\begin{enumerate}
\item if the Euler characteristic  $\chi(A):= \mu \circ \Delta(1) \neq 0$, we have vector space isomorphisms
\begin{equation*}
\HH_{\sg}^i(A, A) \cong \mathrm H^i(\calD^*(A,A)) \cong
\begin{cases}
\HH^i(A, A)& \emph{if $i\leq k-1$,}\\
\HH_{i-k+1}(A, A)  & \emph{if $i\geq k$;}
\end{cases}
\end{equation*}
\item if $\chi(A):= \mu \circ   \Delta(1)=0$, we have vector space isomorphisms
\begin{equation*}
\HH_{\sg}^i(A, A) \cong \mathrm H^i(\calD^*(A,A)) \cong
\begin{cases}
\HH^i(A, A)& \emph{if $i<k-1$,}\\
\HH^{k-1}(A, A)\oplus \HH_0(A, A) & \emph{if $i=k-1$,}\\
\HH_1(A, A)\oplus \HH^k(A, A) & \emph{if $i=k$,}\\
\HH_{i-k+1}(A, A)  & \emph{if $i>k$.}
\end{cases}
\end{equation*}
\end{enumerate}
\end{proposition}

\begin{proof}
The isomorphism $\HH_{\sg}^i(A, A)   \cong \mathrm H^i(\calD^*(A,A))$ is induced by the homotopy retract (\ref{equation-homotopy-transfer}). 

Let us  use the long exact sequence \eqref{align:longexact} to calculate $\mathrm H^i(\calD^*(A,A))$ in terms of Hochschild homology and cohomology. Since $A$ is simply connected, we have $ C^{i}(A, A) = 0$ if  $i> k$ and $C_i(A, A) = 0$ if $i< 0$ by degree reasons. It follows that 
\begin{flalign*}
&&\HH^i(A, A) =\begin{cases}
 0 & \text{if $i>k$}\\ 
A^k\cong \mathbb K& \text{if $i =k$}  
 \end{cases} && \text{and}&& 
 \HH_i(A, A) =\begin{cases}
 0 & \text{if $i<0$}\\ 
A^0 \cong \mathbb K& \text{if $i =0$}.  
 \end{cases} &&
\end{flalign*}
Then by the long exact sequence \eqref{align:longexact} we obtain   
\begin{align*}
\HH_{\sg}^i(A, A) \cong \begin{cases}
 \HH^i(A, A) & \text{if $i<k-1$}\\ 
\HH_{i-k+1}(A, A) & \text{if $i >k$}  
 \end{cases} 
\end{align*}
and   an exact sequence involving $\HH_{\sg}^{k-1}(A, A)$ and $\HH_{\sg}^{k}(A, A)$ 
\begin{align}\label{align:exact}
0 \to \HH^{k-1}(A, A) \to \HH_{\sg}^{k-1}(A, A) \to A^0 \xrightarrow{\gamma} A^k \to \HH_{\sg}^k(A, A) \to \HH_{1}(A, A) \to  0. 
\end{align}

If $\chi(A)  \neq 0$ then the differential $\gamma$ in \eqref{align:exact} is an isomorphism. Hence, we have $$
 \HH_{\sg}^{k-1}(A, A)\cong \HH^{k-1}(A, A) \quad \text{and} \quad \HH_{\sg}^k(A, A)  \cong \HH_1(A, A) .$$
 If $\chi(A)=0$ then the differential $\gamma$ is zero, so we have 
$$
\HH_{\sg}^{k-1}(A, A)\cong  \HH^{k-1}(A, A)\oplus \HH_0(A, A)  \ \text{and} \ \HH_{\sg}^k(A, A) \cong \HH_1(A, A)\oplus \HH^k(A, A).
$$ %\begin{enumerate}
%\item If the Euler characteristic  $\chi(A)  \neq 0$ then the differential $\gamma$ is an isomorphism since $\gamma(1)=\chi(A)\neq 0 \in \mathbb K$.  We infer that both $A^0$ and $A^k$ are killed in cohomology, hence $$
 %\HH_{\sg}^{k-1}(A, A)\cong \HH^{k-1}(A, A) \quad \text{and} \quad \HH_{\sg}^k(A, A)  \cong \HH_1(A, A) .$$ 
%\item \label{itemii} If $\chi(A)=0$ then the differential $\gamma$ is zero, so we have 
%$$
%\HH_{\sg}^{k-1}(A, A)\cong  \HH^{k-1}(A, A)\oplus \HH_0(A, A)  \quad \text{and} \quad \HH_{\sg}^k(A, A) \cong \HH_1(A, A)\oplus \HH^k(A, A).
%$$ 
%\end{enumerate}
\end{proof}

\begin{remark} % In the case \eqref{itemii} above we have 
%$$
%\HH_{\sg}^k (A, A) \cong \HH_1(A, A)\oplus \mathbb K \quad \text{and} \quad \HH_{\sg}^{k-1}(A, A)\cong  \HH^{k-1}(A, A)\oplus \mathbb K.
%$$ 
%In particular, 

It follows from  Remark \ref{remark-simply} and Proposition \ref{proposition4.7} that the (shifted)  reduced Hochschild homology  $s^{1-k}\overline{\HH}_*(A,A)$ is a graded subspace of $\HH_{\sg}^*(A, A)$.  Theorem \ref{theorem-m} below shows that it is actually a graded subalgebra of $\HH_{\sg}^*(A, A)$. 
%For both cases, we have $\HH^{ * > k}(A, A) = 0 = \HH_{* < 0}(A, A)$. 
\end{remark}
\begin{remark} Proposition \ref{proposition4.7} should be compared with the computation of Rabinowitz-Floer homology in Theorem 1.10 of \cite{CiFrOa}. In fact, when $M$ is a simply connected oriented closed manifold of dimension $k$, Proposition \ref{proposition4.7} implies that the singular Hochschild cohomology of the dg algebra of cochains on $M$ with real coefficients is isomorphic to the Rabinowitz-Floer homology of the unit cotangent bundle of $M$. The long exact sequence  \eqref{align:longexact} should be compared with the long exact sequence  in Theorem 1.2 of \cite{CiFrOa}.
\end{remark}
Recall that $\mathbf{DGA}^1_{\mathbb{K}}$ is the category of unital dg $\mathbb{K}$-algebras $A$ which are  simply connected and non-negatively graded (i.e.\ $A^{< 0} = 0, 
 \ A^0 \cong \mathbb{K}$ and $A^1=0$). Our main result, Theorem \ref{main} in the introduction, can now be proved by considering two cases: when the Euler characteristic is zero and when it is non-zero.

\begin{theorem}
\label{theorem-m}
Let $(A, \langle -, -\rangle_A)$ and $(B, \langle -, -\rangle_B)$ be two dg Frobenius algebras of degree $k$ such that $A,B \in \mathbf{DGA}^1_{\mathbb{K}}$. Suppose that there is a zig-zag of quasi-isomorphisms of dg algebras $$A \xleftarrow{\simeq} \bullet \xrightarrow{\simeq} \dotsb \xleftarrow{\simeq} \bullet \xrightarrow{\simeq} B.$$  
\begin{enumerate} 
\item
 If  $\chi(A) \neq 0$, then $\chi(B) \neq 0$ and  the composition of isomorphisms 
$$s^{1-k}\overline{\HH}_*(A,A)  \underset{\mathrm H^*(\iota)}{\xrightarrow{\cong}}  \HH^{*\geq k}_{\sg}(A,A)  \cong \HH^{*\geq k}_{\sg}(B,B) \underset{\mathrm H^*(\iota)^{-1}}{\xrightarrow{\cong}}  s^{1-k}\overline{\HH}_*(B,B)$$
preserves the Goresky-Hingston algebra structures.

\item If $\chi(A)=0$, then $\chi(B) =0$ and the composition of isomorphisms $$s^{1-k}\overline{\HH}_*(A,A) \oplus\HH^k(A,A) \underset{\mathrm H^*(\iota)}{\xrightarrow{\cong}} \HH^{*\geq k}_{\sg}(A,A)  \cong \HH^{*\geq k}_{\sg}(B,B) \underset{\mathrm H^*(\iota)^{-1}}{\xrightarrow{\cong}} s^{1-k}\overline{\HH}_*(B,B) \oplus\HH^k(B,B)$$ restricts to an algebra  isomorphism $$s^{1-k}\overline{\HH}_*(A,A) \cong s^{1-k}\overline{\HH}_*(B,B).$$ %which preserves the Goresky-Hingston algebra structures. 
\end{enumerate}
\end{theorem}

\begin{proof}
First observe that $\chi(A)\neq 0$ if and only if $\chi(B)\neq 0$. This follows, for instance, from the quasi-isomorphism invariance of (singular) Hochschild cohomology and Proposition \ref{proposition4.7}, which tells us that  if $\chi(A) \neq 0$ then $\HH_{\sg}^{k-1}(A,A)  \cong \HH^{k-1}(A,A)$ and if $\chi(A)=0$ then $\HH_{\sg}^{k-1}(A,A) \cong  \HH^{k-1}(A,A)\oplus \mathbb K$.

%Recall from Remark \ref{remark-simply} that for any simply connected (non-negatively graded) dg algebra $A$, $\overline{\HH}_*(A,A)=\HH_{*>0}(A,A)$. 
In both cases (1) and (2), the middle isomorphism is obtained from the invariance of the singular Hochschild cohomology algebra along quasi-isomorphisms, as proven in Proposition \ref{proposition-7.4}. The outer isomorphisms of graded algebras are obtained from Propositions \ref{proposition4.7} and  \ref{prop:iotastar} and Remark \ref{remark-simply}. It follows that the compositions in both cases (1) and (2) are isomorphisms of graded algebras. %Hence,  in both cases (1) and (2) the first and last isomorphisms of graded vector spaces follow from Proposition \ref{proposition4.7}. The middle isomorphisms $ \HH^{*\geq k}_{\sg}(A,A)  \cong \HH^{*\geq k}_{\sg}(B,B)$ follows from Proposition \ref{proposition-7.4}. By Proposition  \ref{prop:iotastar}, the compositions in both cases are isomorphisms of graded algebras. 
It remains to show that the composition in the case (2)
%$$s^{1-k}\overline{\HH}_*(A,A) \oplus\HH^k(A,A) \underset{\mathrm H^*(\iota)}{\xrightarrow{\cong}} \HH^{*\geq k}_{\sg}(A,A)  \cong \HH^{*\geq k}_{\sg}(B,B) \underset{\mathrm H^*(\iota)^{-1}}{\xrightarrow{\cong}} s^{1-k}\overline{\HH}_*(B,B) \oplus\HH^k(B,B)$$
 restricts to an isomorphism $$s^{1-k}\overline{\HH}_*(A,A) \cong s^{1-k}\overline{\HH}_*(B,B).$$ 
 This follows from the functoriality of the long exact sequences with respect to quasi-isomorphisms in Corollary \ref{corollary:longexactsequence}. 
\end{proof}

\begin{remark}
The argument in the above proof actually shows that, if $\chi(A)=0=\chi(B)$, the isomorphism $\HH_{\sg}^*(A,A) \cong \HH_{\sg}^*(B,B)$ restricts to an isomorphism of (ordinary) Hochschild homology $s^{1-k}\HH_*(A, A)\cong s^{1-k}\HH_*(B, B)$ which preserves the Goresky-Hingston algebra structures, which exist without restricting to reduced Hochschild homology in this particular case as mentioned in Remark \ref{extension}. 
\end{remark}

\begin{proof}[Proof of Theorem \ref{main}] Theorem \ref{main} in the introduction follows directly from Theorem \ref{theorem-m}.
\end{proof}

\begin{proof}[Proof of Corollary \ref{corollary1.2}]  Part (1) of the corollary follows from Theorem \ref{main} since any two Poincar\'e duality models for $\mathcal{A}(M)$ are connected by a zig-zag of quasi-isomorphisms of simply connected cdg algebras. Part (2) follows since if $A$  and $A'$ are Poincar\'e duality models for $\mathcal{A}(M)$ and $\mathcal{A}(M')$ and $M$ and $M'$ are homotopy equivalent oriented closed manifolds of dimension $k$, then $A$ and $A'$ are connected by a zig-zag of quasi-isomorphisms of simply connected cdg algebras. \end{proof}

\appendix
\section{Proof of Theorem \ref{thm:homtopyretract}}
\label{sectionappendix}
In this appendix we will prove Theorem \ref{thm:homtopyretract}. 
\subsection{Basic identities regarding the action $\blacktriangleright$}

The following identities will be useful when keeping track of the signs in the computations in this appendix. 

\begin{lemma}\label{lemma:bascipropertyeifi1}
Let $\partial_h$ be the external differential of $C_*(A, A)$ as in Definition \ref{definition:hochschildchain}. Let  $\blacktriangleright$ be the action described in Lemma \ref{lemma-bimodule}. Define $\varepsilon(\overline{a}) = \langle a, 1\rangle$ for any $a \in A$.
For any element $\alpha =\overline{a_1}\otimes \dotsb \otimes \overline{a_p} \otimes a_{p+1} \in (s\overline A)^{\otimes p} \otimes A$, we have the following identities
\begin{enumerate}
\item $ \sum_i (-1)^{|f_i||\alpha|} (\varepsilon \otimes \id^{\otimes p}) (e_i \blacktriangleright (\overline{a_1}\otimes \dotsb \otimes \overline{a_p} \otimes a_{p+1}  f_i)) = \partial_h (\alpha),$
\item  $\sum_i (-1)^{(|f_i|-1)(1-k)} \overline{e_i}\otimes (\varepsilon\otimes \id^{\otimes p})(f_i\blacktriangleright \alpha) = -\alpha,$
\item $\sum_i (-1)^{|e_i|(1-k)} e_i \blacktriangleright (\varepsilon \otimes \id^{\otimes p}) (f_i \blacktriangleright \alpha) = 0$.

\end{enumerate}

\end{lemma}
\begin{proof}
For the first identity,  we have 
\begin{align*}
&\sum_i (-1)^{|f_i||\alpha|} (\varepsilon \otimes \id^{\otimes p}) (e_i \blacktriangleright (\overline{a_1}\otimes \dotsb \otimes \overline{a_p} \otimes a_{p+1} f_i))\\
={}&   \sum_i(-1)^{|f_i||\alpha|+|e_i|}  \bigg( \varepsilon(e_ia_1) \overline{a_2} \otimes \dotsb\otimes \overline{a_p} \otimes a_{p+1}f_i  -(-1)^{ \epsilon_{p-1}} \varepsilon(e_i) \overline{a_1}\otimes \dotsb\otimes \overline{a_{p-1}} \otimes a_p a_{p+1}f_i\\
&\quad\quad + \sum_{j=1}^{p-1} (-1)^{\epsilon_j} \varepsilon(e_i) \overline{a_1}\otimes\dotsb\otimes \overline{a_{j-1}} \otimes \overline{a_ja_{j+1}}\otimes \dotsb \otimes \overline{a_p} \otimes a_{p+1}f_i\bigg)\\
={}& \partial_h(\overline{a_1}\otimes \dotsb \otimes \overline{a_p}\otimes a_{p+1}),
\end{align*}
where $\epsilon_j = |a_1| +\dotsb+|a_j|-j$ and the second equality follows from Lemma \ref{lemma:bascipropertyeifi} (1).

Let us verify the second identity.  We have 
\begin{align*}
&\sum_i (-1)^{(|f_i|-1)(1-k)} \overline{e_i}\otimes (\varepsilon\otimes \id^{\otimes p})(f_i\blacktriangleright (\overline{a_1}\otimes \dotsb \otimes \overline{a_p} \otimes a_{p+1}) )\\
={}& \sum_i \bigg((-1)^{(|f_i|-1)(1-k) +|f_i|} \varepsilon(f_ia_1) \overline{e_i } \otimes \overline{a_2}\otimes  \dotsb \otimes \overline{a_p} \otimes a_{p+1} \\
&\ \ + \sum_{j=1}^{p-1}(-1)^{(|f_i|-1)(1-k) + |f_i| + \epsilon_j} \varepsilon(f_i) \overline{e_i } \otimes \overline{a_1} \otimes \dotsb \otimes \overline{a_{j-1}} \otimes \overline{a_ja_{j+1}} \otimes \dotsb\otimes a_{p+1} \\
&\ \ - (-1)^{(|f_i|-1)(1-k) + |f_i| + \epsilon_{p-1}} \varepsilon(f_i) \overline{e_i } \otimes  \overline{a_1}\otimes \dotsb\otimes \overline{a_{p-1}} \otimes a_pa_{p+1}\bigg)\\
={}&- \overline {a_1} \otimes \dotsb \otimes \overline{a_p} \otimes a_{p+1},
\end{align*}
where in the second equality we use the fact that  $\sum_i \varepsilon(f_i) \overline{e_i} = \overline 1 = 0$ in $s\overline A$. The third identity can be verified in a similar way. 
\end{proof}

\subsection{The injection $\iota$ is a cochain map}
Recall the injection $\iota \colon \calD^*(A, A) \hookrightarrow \calC_{\sg}^*(A,A)$ is defined in Definition \ref{definitioniota}. We now check $\iota$ is compatible with differentials. 

\begin{lemma}\label{lemma:appendix}
The map $\iota$ is a cochain map. 
\end{lemma}
\begin{proof}
Let us first check that $\iota$ is compatible with the external differentials. It suffices to prove that the following diagram commutes for $p >1$
\[
\xymatrix{
s^{1-k}C_{-(p-1), *} (A, A) \ar[d]_-{(-1)^{1-k} \partial_h} \ar[r]^-{\iota} & C^{0, *}(A, \Omega_{\nc}^p(A))\ar[dr]^-{\delta^h}\\
s^{1-k}C_{-(p-2), *}(A, A) \ar[r]^-{\iota} & C^{0, *}(A, \Omega_{\nc}^{p-1}(A)) \ar[r]^-{\theta} & C^{1, *}(A, \Omega_{\nc}^{p}(A)).  
}
\]
Let $\alpha = \overline{a_1} \otimes \dotsb \otimes \overline{a_{p-1}} \otimes a_{p}$. We have
 \begin{align*}
&(-1)^{1-k}\theta\circ \iota \circ \partial_h( \overline{a_1} \otimes \dotsb \otimes \overline{a_{p-1}} \otimes a_{p}) (\overline{a_0})\\
 ={}&  \sum_i \sum_{j=1}^{p-2}(-1)^{1-k+ |f_i|(|\alpha|+1) + (|a_0|-1)(|\alpha|+k) + \epsilon_j} \overline{a_0} \otimes \overline{e_i} \otimes \overline{a_1} \otimes \dotsb \otimes \overline{a_ja_{j+1}} \otimes \dotsb \otimes \overline{a_{p-1}} \otimes a_p f_j\\
& -\sum_i (-1)^{1-k+ |f_i|(|\alpha|+1) + (|a_0|-1)(|\alpha|+k) + \epsilon_{p-2}} \overline{a_0} \otimes \overline{e_i} \otimes \overline{a_1}\otimes \dotsb \otimes \overline{a_{p-2}}\otimes a_{p-1}a_p f_i\\
&+ \sum_i (-1)^{1-k+ |f_i|(|\alpha|+1) + (|a_0|-1)(|\alpha|+k) + (|a_2|+\dotsb+|a_{p}|-p)|a_1|} \overline{a_0} \otimes \overline{e_i} \otimes \overline{a_2}\otimes \dotsb \otimes \overline{a_{p-1}}\otimes a_pa_1 f_i
\end{align*}
and  
\begin{align*}
\delta^h \circ \iota( \overline{a_1} \otimes \dotsb \otimes \overline{a_{p-1}} \otimes a_{p}) (\overline{a_0})=& -\sum_i(-1)^{|f_i||\alpha| +(|a_0|-1)(|\alpha|+k-1)} a_0 \blacktriangleright  (\overline{e_i} \otimes \overline{a_1}\otimes \dotsb \otimes a_{p}f_i) \\
& +\sum_i(-1)^{|f_i||\alpha| + |\alpha|+k-1 } \overline{e_i} \otimes \overline{a_1}\otimes \dotsb \otimes \overline{a_{p-1}} \otimes a_{p}f_i a_0,
\end{align*}
where $\epsilon_j = |a_1|+\dotsb+|a_j|-j$.
By Lemma \ref{lemma-bimodule}, we may cancel the common terms and obtain
\begin{align*}
&((-1)^{1-k}\theta\circ \iota \circ \partial_h-\delta^h \circ \iota)( \overline{a_1} \otimes \dotsb \otimes \overline{a_{p-1}} \otimes a_{p}) (\overline{a_0})  
\\
={}& \sum_i \bigg((-1)^{1-k+ |f_i|(|\alpha|+1) + (|a_0|-1)(|\alpha|+k) + (|a_2|+\dotsb+|a_{p}|-p)|a_1|} \overline{a_0} \otimes \overline{e_i} \otimes \overline{a_2}\otimes \dotsb \otimes \overline{a_{p-1}}\otimes a_pa_1 f_i\\
&\quad\quad + (-1)^{|f_i||\alpha| +(|a_0|-1)(|\alpha|+k-1)+|a_0|} \overline{a_0e_i} \otimes \overline{a_1}\otimes \dotsb \otimes \overline{a_{p-1}} \otimes a_{p}f_i \\
&\quad \quad+ (-1)^{|f_i||\alpha| +(|a_0|-1)(|\alpha|+k-1)+|a_0|+|e_i|-1} \overline{a_0} \otimes \overline{e_ia_1}\otimes \dotsb \otimes \overline{a_{p-1}} \otimes a_{p}f_i\\
& \quad\quad -(-1)^{|f_i||\alpha| + |\alpha|+k-1 } \overline{e_i} \otimes \overline{a_1}\otimes \dotsb \otimes \overline{a_{p-1}} \otimes a_{p}f_i a_0\bigg).
\end{align*} 
It follows from  Remark \ref{remark:usefulidentities} that the first term cancels with the third one and the second term cancels with the fourth one. We obtain $(-1)^{1-k}\theta\circ \iota \circ \partial_h-\delta^h \circ \iota=0$. Similarly, we may check that $\iota$ is compatible with the internal differential. \end{proof}

\subsection{The surjection $\Pi$} 

We will now construct $\Pi\colon \calC_{\sg}^*(A, A)\rightarrow \calD^*(A, A)$.  If  $m,p \in \Z_{>0}$ define   $$\pi_{m, p}\colon C^{m, *}(A, \Omega_{\nc}^p(A))\rightarrow C^{m-1, *}(A, \Omega_{\nc}^{p-1}(A))$$ on any $f\in C^{m, *}(A, (\sA)^{\otimes p}\otimes A) =  \Hom_{\mathbb K}((\sA)^{\otimes m}, (\sA)^{\otimes p} \otimes A)$ by letting
%$$\pi_{m,p}(f) := \sum_i \blacktriangleright \circ (\id \otimes \epsilon \otimes \id^{\otimes p}) \circ (\id \otimes f) \circ (\id \otimes \pi \otimes \id^{\otimes m-1}) \circ (e_i \otimes f_i \otimes \id^{m-1}),$$
%where $\epsilon: s\overline{A} \to \mathbb{K}$ is the degree $k-1$ map given by $\epsilon( \overline{a} ) = \langle a,1\rangle$ and  $\blacktriangleright : A \otimes \Omega_{\nc}^p(A) \to \Omega_{\nc}^p(A)$ is the action is defined in Lemma \ref{lemma-bimodule}. Also recall $\pi: A \to s\overline{A}$ is the degree $-1$ projection map. Applying this definition to an element  $\overline{a_1}\otimes \cdots\otimes \overline{a_{m-1}} \in (s\overline{A})^{\otimes m-1}$ we obtain
\begin{align*}
 \pi_{m,p}(f)(\overline{a_1}\otimes \cdots\otimes \overline{a_{m-1} })=\sum_{i } (-1)^{ (|f_i|-1) (|f|+k)}  e_i \blacktriangleright(\varepsilon \otimes \id^{\otimes p}) (f(\overline{f_i}\otimes \overline{a_1}\otimes \cdots \otimes \overline{a_{m-1}})),
%={} &\sum_{i }(-1)^{\deg(f) (\deg(f_i)-1) + \deg(e_i)(k-1)}  (\pi\otimes \id^{\otimes p-1 })b_{-p+1}\big(e_i\otimes  (\epsilon\otimes \id^{\otimes p}) (f(\overline{f_i}\otimes \overline{a_1}\otimes \cdots \otimes\overline{a_{m-1}}))\big),
\end{align*}
where $\varepsilon\colon s\overline{A} \to \mathbb{K}$ is the degree $1-k$ map given by $\varepsilon( \overline{a} ) = \langle a,1\rangle$ and  $\blacktriangleright$ is the left action of $A$ on $\Omega^{p-1}_{\nc}(A)$ defined in Lemma \ref{lemma-bimodule}. For convenience, set $\pi_{m,0}=\id\colon C^{m,*}(A,A) \to C^{m,*}(A,A)$ for $m \geq 0$. %Also recall $\pi\colon A \to s\overline{A}$ is the degree $-1$ projection map.

 Define  $$\pi_{0, p}\colon C^{0, *}(A, \Omega_{\nc}^p(A)) \rightarrow C_{-(p-1), *}(A, A), \quad\quad \text{for $p > 0$}$$ as follows. If $f \in C^{0, *}(A, \Omega_{\nc}^p(A)) = \text{Hom}_{\mathbb{K}}(\mathbb{K}, \Omega_{\nc}^p(A))$ and $f(1)=\overline{a_1}\otimes \cdots \otimes \overline{a_p}\otimes a_{p+1} \in \Omega_{\nc}^p(A)=(\sA)^{\otimes p}\otimes A$ let $$\pi_{0, p}(f)= (-1)^{ k}\varepsilon(\overline{a_1})\overline{a_2}\otimes \cdots \otimes \overline{a_p}\otimes a_{p+1}.$$

%Let us verify the third identity. We have 
%\begin{align*}
%& \sum_i (-1)^{|e_i|(1-k)} e_i \blacktriangleright (\varepsilon \otimes \id^{\otimes p}) (f_i \blacktriangleright (\overline{a_1}\otimes \dotsb \otimes \overline{a_p} \otimes a_{p+1}))\\
%={}& \sum_i (-1)^{|e_i|(1-k) +|f_i|} \varepsilon(f_ia_1) e_i \blacktriangleright  (\overline{a_2}\otimes  \dotsb \otimes \overline{a_p} \otimes a_{p+1} )\\
%&+ \sum_i \sum_{j=1}^{p-1}(-1)^{|e_i|(1-k)+|f_i|+\epsilon_j} \varepsilon(f_i) e_i \blacktriangleright  (\overline{a_1} \otimes \dotsb \otimes \overline{a_{j-1}} \otimes \overline{a_ja_{j+1}} \otimes \dotsb\otimes a_{p+1}) \\
%&- \sum_i (-1)^{|e_i|(1-k) + |f_i|+ \epsilon_{p-1}} \varepsilon(f_i) e_i \blacktriangleright ( \overline{a_1}\otimes \dotsb\otimes \overline{a_{p-1}} \otimes a_pa_{p+1})\\
%={}&  (-1)^k a_1 \blacktriangleright (\overline{a_2} \otimes \dotsb \otimes \overline{a_p} \otimes a_{p+1}) - (-1)^k a_1 \blacktriangleright (\overline{a_2} \otimes \dotsb \otimes \overline{a_p} \otimes a_{p+1})\\
%={}& 0, 
%\end{align*}
%where the second identity follows from Lemma \ref{lemma:bascipropertyeifi}: the first sum on the right hand side of the second identity equals $(-1)^k a_1 \blacktriangleright (\overline{a_2} \otimes \dotsb \otimes \overline{a_p} \otimes a_{p+1})$ and the remaining two sums equal $-(-1)^k a_1 \blacktriangleright (\overline{a_2} \otimes \dotsb \otimes \overline{a_p} \otimes a_{p+1})$.

\begin{lemma}\label{lemma5.2}
$\pi_{*, *}$ is compatible with the differentials.
\end{lemma}
\begin{proof}
First we check $\pi_{>0, *}$ is compatible with the external differential $\delta^h$. Let $f\in C^{m+1, *}(A, \Omega_{\nc}^p(A))$, $m\geq 0$. Then for any $\overline{a_1}\otimes \cdots \otimes \overline{a_{m+1}}\in 
(\sA)^{\otimes m+1}$ we have 

\begin{align*}
&\pi_{*, *}\circ \delta^h(f) (\overline{a_1}\otimes \cdots \otimes \overline{a_{m+1}}) \\
={} &-\sum_i (-1)^{(|f_i|-1)(k+1)} e_i \blacktriangleright(\varepsilon \otimes \id^{\otimes p}) (f_i\blacktriangleright f ( \overline{a_1}\otimes \cdots \otimes
\overline{a_{m+1}}))\\
&-\sum_i  (-1)^{|f_i||f|+(|f_i|-1)k}e_i \blacktriangleright(\varepsilon \otimes \id^{\otimes p}) (f(\overline{f_ia_1}\otimes  \cdots \otimes \overline{a_{m+1}}))\\
& -\sum_{j=1}^m\sum_i (-1)^{|f_i||f| +(|f_i|-1)k+ \epsilon_j} e_i \blacktriangleright(\varepsilon \otimes \id^{\otimes p}) (f(\overline{f_i}\otimes  \cdots \otimes  \overline{a_ja_{j+1}}\otimes \cdots \otimes 
\overline{a_{m+1}}))\\
 & +\sum_i (-1)^{|f_i||f| +(|f_i|-1)k+ \epsilon_m} e_i \blacktriangleright(\varepsilon \otimes \id^{\otimes p}) (f(\overline{f_i}\otimes \overline{a_1}\otimes \cdots \otimes \overline{a_m})a_{m+1}),
\end{align*}
where $\epsilon_j = |a_1|+\dotsb+ |a_j|-j$.
On the other hand, we have
\begin{align*}
&\delta^h\circ \pi_{*, *}(f) (\overline{a_1}\otimes \cdots \otimes \overline{a_{m+1}})\\
%={}&
%\sum_{j=1}^m \pm \pi_{*, *}(f)(\overline{a_1}\otimes \cdots \otimes \overline{a_ja_{j+1}}\otimes \cdots \otimes \overline{a_{m+1}})\\
%&\pm a_1 \blacktriangleright(\pi_{*, *}(f) (\overline{a_2}\otimes \dotsc \otimes \overline{a_{m+1}})) \pm (\pi_{*, *}(f) (\overline{a_1}\otimes \dotsc \otimes \overline{a_m})) a_{m+1}\\
={}&-\sum_i (-1)^{(|a_1|+|f_i|)|f|+(|f_i|-1)k} a_1\blacktriangleright( e_i \blacktriangleright(\varepsilon \otimes \id^{\otimes p}) (f(\overline{f_i}\otimes \overline{a_2}\otimes \cdots \otimes
\overline{a_{m+1}})))\\
&-\sum_{j=1}^{m}\sum_i (-1)^{|f_i||f| +(|f_i|-1)k+ \epsilon_j} e_i \blacktriangleright(\varepsilon \otimes \id^{\otimes p}) (f(\overline{f_i}\otimes \cdots \otimes \overline{a_ja_{j+1}}\otimes \cdots \otimes
\overline{a_{m+1}}))\\
&+ \sum_i (-1)^{|f_i||f| +(|f_i|-1)k+ \epsilon_m}  e_i \blacktriangleright(\varepsilon \otimes \id^{\otimes p}) (f(\overline{f_i}\otimes \overline{a_1}\otimes \cdots \otimes
\overline{a_{m}})a_{m+1}).
\end{align*}

Thus, we may cancel the common terms to obtain:
\begin{align}
\label{align7}
&(\pi_{*, *}\circ \delta^h-\delta^h\circ \pi_{*, *})(f)(\overline{a_1}\otimes \cdots \otimes \overline{a_{m+1}})
\\
={}&  -\sum_i (-1)^{(|f_i|-1)(k+1)} e_i \blacktriangleright(\varepsilon \otimes \id^{\otimes p}) (f_i\blacktriangleright f ( \overline{a_1}\otimes \cdots \otimes
\overline{a_{m+1}})) \nonumber\\
&-\sum_i  (-1)^{|f_i||f|+(|f_i|-1)k}e_i \blacktriangleright(\varepsilon \otimes \id^{\otimes p}) (f(\overline{f_ia_1}\otimes  \cdots \otimes \overline{a_{m+1}}))\nonumber\\
& +\sum_i (-1)^{(|a_1|+|f_i|)|f|+(|f_i|-1)k} a_1\blacktriangleright( e_i \blacktriangleright(\varepsilon \otimes \id^{\otimes p}) (f(\overline{f_i}\otimes \overline{a_2}\otimes \cdots \otimes
\overline{a_{m+1}}))).\nonumber
\end{align}
 It  follows from  Lemma  \ref{lemma:bascipropertyeifi1} (3) that  the first sum on the right hand side of the equality in (\ref{align7}) vanishes.
Since $\blacktriangleright$ defines a left action of $A$ on $\Omega_{\nc}^p(A)$, we have 
\begin{align*}
a_1\blacktriangleright( e_i \blacktriangleright(\varepsilon \otimes \id^{\otimes p}) (f(\overline{f_i}\otimes \overline{a_2}\otimes \cdots \otimes\overline{a_{m+1}})))= (a_1 e_i) \blacktriangleright(\varepsilon \otimes \id^{\otimes p}) (f(\overline{f_i}\otimes \overline{a_2}\otimes \cdots \otimes\overline{a_{m+1}})),
\end{align*}
thus the second sum on the right hand side in (\ref{align7}) cancels with the third sum since $\sum a_1e_i\otimes f_i=\sum (-1)^{|a_1|k} e_i \otimes f_i a_1$.  % It  follows from  Lemma  \ref{lemma:bascipropertyeifi1} that  the first sum vanishes, namely
 %\begin{equation*}
%\begin{split}
% -\sum_i (-1)^{(|f_i|-1)} e_i \blacktriangleright(\varepsilon \otimes \id^{\otimes p}) (f_i\blacktriangleright f ( \overline{a_1}\otimes \cdots \otimes
%\overline{a_{m+1}})) = 0. 
%\end{split}
%\end{equation*}
Therefore, $(\pi_{*, *}\circ \delta^h-\delta^h\circ \pi_{*, *})(f)=0$.  Similarly, we may check that $\pi_{>0, *}$ is compatible with the internal differential $\delta^v$, which we leave to the reader.

We now check that $\pi_{0, *}$ are compatible with the external differentials $\delta^h$ and $\partial_h$. Namely, that the following diagram commutes for any $p\in \Z_{> 0}$:
\begin{equation}\label{equation-diagram2}
\xymatrix{
C^{0, *}(A, \Omega_{\nc}^p(A))\ar[d]^-{\delta^h} \ar[r]^-{\pi_{0, p}}  & s^{1-k} C_{-(p-1), *}(A, A)\ar[rd]^-{(-1)^{1-k}\partial_h}\\
C^{1, *}(A, \Omega_{\nc}^p(A)) \ar[r]^-{\pi_{1, p}} & C^{0, *}(A, \Omega_{\nc}^{p-1}(A)) \ar[r]^-{\pi_{0, p-1}}& s^{1-k} C_{-(p-2), *}(A, A).
}
\end{equation}
The commutativity of diagram (\ref{equation-diagram2}) follows since
%\begin{equation*}
%\begin{split}
%\partial_h\circ \pi_{0, p}(\overline{a_1}\otimes \cdots \overline{a_p}\otimes a_{p+1}) ={}& 
%- \sum_{j=2}^{p-1} (-1)^{|a_1|(|x|-1)+\epsilon_j} \varepsilon(a_1) \overline{a_2}\otimes \dotsb \otimes \overline{a_ja_{j+1}} \otimes \dotsb \otimes \overline{a_p}\otimes a_{p+1}\\
%& +(-1)^{|a_1|(|x|-1) + \varepsilon_{p-1}} \varepsilon(a_1) \overline{a_2} \otimes \dotsb\otimes \overline{a_{p-1}} \otimes a_p a_{p+1}\\
%&+(-1)^{|a_1||x| + |a_2|(|a_3|+\dotsb+|a_{p+1}|-p)} \varepsilon(a_1) \overline{a_3} \otimes \dotsb\otimes \overline{a_p} \otimes a_{p+1}a_2
% \end{split}
%\end{equation*}
%and 
\begin{align*}
&\pi_{0, p-1}\circ \pi_{1, p}\circ \delta^h(\overline{a_1}\otimes \cdots \otimes \overline{a_p}\otimes a_{p+1})\\
={}&- \sum_i (-1)^{(|f_i|-1)(k-1)} \pi_{0, p-1}( e_i  \blacktriangleright (\varepsilon\otimes \id^{\otimes p}) (f_i \blacktriangleright (\overline{a_1}\otimes \dotsb\otimes \overline{a_p} \otimes a_{p+1})))\\
& + \sum_i (-1)^{(|f_i|-1)(|\alpha|+k+1) + |\alpha|} \varepsilon(a_1)  \pi_{0, p-1} ( e_i \blacktriangleright(\overline{a_2}\otimes \dotsb\otimes \overline{a_p} \otimes a_{p+1}f_i))\\
={}&\sum_i (-1)^{(|f_i|-1)(|\alpha|+k+1) + |\alpha|} \varepsilon(a_1) \pi_{0, p-1}  ( e_i \blacktriangleright(\overline{a_2}\otimes \dotsb\otimes \overline{a_p} \otimes a_{p+1}f_i))\\
={} &(-1)^{1-k} \partial_h\circ \pi_{0, p}(\overline{a_1} \otimes \dotsb\otimes \overline{a_{p+1}})
%={} & \sum_i \varepsilon(\overline{a_1})\overline{a_3}\otimes \cdots \otimes \overline{a_p}\otimes a_{p+1}a_2 +\sum_i \sum_{j=1}^{p-1} \pm \varepsilon(\overline{a_1})\overline{a_2}\otimes \cdots \otimes \overline{a_{j+1}a_{j+2}}\otimes \cdots \otimes a_{p+1}f_i
%={} & \partial_h\circ \pi_{0, p}(\overline{a_1}\otimes \cdots \otimes \overline{a_p}\otimes a_{p+1}),
\end{align*}
where the second identity follows from Lemma \ref{lemma:bascipropertyeifi1} (3) and the third identity from Lemma \ref{lemma:bascipropertyeifi1} (1). Similarly, we may check that $\pi_{0, *}$ are compatible with the internal differentials.
\end{proof}

\begin{lemma}\label{lemma5.3}
 We have the following identities 
\begin{flalign*}
&&\pi_{m+1, p+1}\circ \theta_{m, p}&=\id &&\text{for $m, p \in \Z_{\geq 0}$}&&\\
&&\pi_{0, p+1}\circ \iota&=\id && \text{for $ p\in \Z_{\geq  0}$}.&&
\end{flalign*} 
\end{lemma}

\begin{proof}
Recall that $\theta_{m, p}$ is defined in (\ref{equation-definition-theta}). 
%for any $f\in C^{m-1, *}(A, (\sA)^{\otimes p-1}\otimes A)$, $$\theta_{m-1, p-1}\colon C^{m-1, *}(A, \Omega_{\nc}^{p-1}(A))\rightarrow C^{m, *}(A, \Omega_{\nc}^p(A))$$ is defined by $$\theta_{m-1,p-1}(f)(\overline{a_1}\otimes \cdots\otimes \overline{a_m}):= (-1)^{(\text{deg}(a_1) -1)\text{deg}( f) } \overline{a_1}\otimes f(\overline{a_2}\otimes \cdots \otimes \overline{a_m}).$$
We have
\begin{align*}
&\pi_{m+1, p+1}\circ \theta_{m, p}(f) (\overline{a_1}\otimes\cdots \otimes  \overline{a_{m}})\\
={}&\sum_i (-1)^{(|f_i|-1)(|f|+k)} e_i \blacktriangleright(\varepsilon \otimes \id^{\otimes p+1}) (\theta_{m, p}(f) (\overline{f_i}\otimes \overline{a_1}\otimes\cdots \otimes  \overline{a_{m}}))
%\\={}&(-1)^{(|a_1|-1)|f|}\pi_{m+1, p+1}(\overline{a_1} \otimes f (\overline{a_1}\otimes\cdots \otimes  \overline{a_{m}}) \\
\\={} &\sum_i  (-1)^{(|f_i|-1)k} e_i \blacktriangleright(  \varepsilon(\overline{f_i}) f(\overline{a_1}\otimes \cdots \otimes
\overline{a_{m-1}}))\\
={}&f(\overline{a_1}\otimes \cdots\otimes \overline{a_{m-1}}),
\end{align*}
where the last identity follows from the fact that $\sum_i   e_i  \varepsilon(\overline{f_i}) =\sum_i e_i \langle f_i, 1\rangle =  1$; see Lemma \ref{lemma:bascipropertyeifi}. 

Similarly, let $\alpha=\overline{a_1}\otimes \cdots \otimes \overline{a_{p}}\otimes a_{p+1} \in C_{-p, *}(A, A),$ then we have 
\begin{equation*}
\begin{split}
\pi_{0, p+1}\circ \iota(\overline{a_1}\otimes \cdots \otimes \overline{a_{p}}\otimes a_{p+1})& =\sum_i  (-1)^{|f_i||\alpha|}\pi_{0, p+1} (\overline{e_i}\otimes
\overline{a_1}\otimes \cdots \otimes \overline{a_{p}}\otimes a_{p+1} f_i)\\
&=\sum_i (-1)^{|f_i||\alpha| + k} \varepsilon(\overline{e_i}) \overline{a_1}\otimes \cdots \overline{a_{p}} \otimes a_{p+1}f_i\\
&=\overline{a_1}\otimes \cdots \otimes \overline{a_{p}}\otimes a_{p+1},
\end{split}
\end{equation*}
where the last identity follows from Lemma \ref{lemma:bascipropertyeifi}.
\end{proof}

\begin{definition}
\label{defn-pi}
Define $\Pi\colon \calC^*_{\sg}(A, A)\rightarrow \calD^*(A, A)$ as follows. For an  element $\widetilde f \in \calC^*_{\sg}(A, A)$ represented by  $f\in C^{m, *}(A, \Omega_{\nc}^p(A))$, let
\begin{equation*}
\Pi(\widetilde f) := 
\begin{cases}
\pi_{m-p, 0} \circ \cdots\circ \pi_{m, p}(f) & \emph{if $m-p\geq 0$},\\
\pi_{0, p-m}\circ \pi_{1, m-p+1}\circ\cdots \circ \pi_{m, p}(f) & \emph{if $m-p<0$}.
\end{cases}
\end{equation*}
Lemma \ref{lemma5.3} implies that this is indeed well-defined, namely, $\Pi$ does not depend on the representative of $\widetilde{f}$.  Moreover, it follows from Lemmas \ref{lemma5.2} and  \ref{lemma5.3} that $\Pi$ is a morphism of cochain complexes such  that $\Pi\circ \iota=\id$. 
\end{definition}

\subsection{The chain homotopy $H$ and Theorem \ref{thm:homtopyretract}} We shall now define the chain homotopy  $H \colon \calC_{\sg}^*(A, A)\rightarrow \calC_{\sg}^{*-1}(A, A)$. Suppose $m, p\in \Z_{>0}$. Given any  $f\in C^{m, *}(A, (\sA)^{\otimes p}\otimes A) $ define  $$h_{m, p}(f) \in C^{m-1, *}(A, (\sA)^{\otimes p}\otimes A)$$ by
%$$h_{m,p}(f) = \sum_i (\id \otimes \varepsilon \otimes \id^{\otimes p} ) \circ (\id \otimes f) \circ (\pi \otimes \pi \otimes  \id^{\otimes m-1} ) \circ (e_i \otimes f_i \otimes \id^{\otimes m-1} ),$$ i.e.
$$\!\!\!\! h_{m,p}(f) (\overline{a_1}\otimes \cdots \otimes \overline{a_{m-1}}) =  \sum_{i}(-1)^{(|f_i|-1)(|f|+k)}  \overline{e_i}\otimes (\varepsilon\otimes \id^{\otimes p}) (f(\overline{f_i}\otimes  \overline{a_1}\otimes \cdots \otimes \overline{a_{m-1}})).$$
We also define $h_{m, p}:=0$ if either $m\leq 0$ or $p\leq 0$. Note that $h_{m,p}$ is of degree $-1$. It follows from the identity $\sum_i e_i \varepsilon (f_i) = 1$, that  $h_{m, p} \circ \theta_{m-1, p-1} =0$.

\begin{lemma}\label{lemma:pimp}
%for $m\in\Z_{\geq 0}$ and $p\in \Z_{>0}$
For any $p >0$ we have the following identities 
\begin{equation}\label{equation5.6}
\delta\circ h_{m, p}+h_{m+1, p} \circ\delta =\begin{cases}
\id-\theta_{m-1, p-1}\circ \pi_{m,p} & \text{if $m >0$}\\
\id - \iota \circ \pi_{0, p} & \text{if $m =0$.}
\end{cases}
\end{equation}

\end{lemma}

\begin{proof}
By a similar computation in the proof of Lemma \ref{lemma5.2}, we can show that $h_{*, *}$ commutes with the internal differentials (i.e. $\delta^v\circ h_{*, *} + h_{*, *} \circ \delta^v = 0$). %Thus, to prove (\ref{equation5.6}) is equivalent to %prove $$\id-\theta_{m-1, p-1}\circ \pi_{m,p}=\delta^h\circ h_{m, p}+h_{m+1, p} \circ\delta^h.$$ 
Thus, it is sufficient to 
prove that we have the following homotopy diagram
\begin{equation*}
\xymatrix{
0\ar[r] & C^{0,*}(A, \Omega_{\nc}^p(A)) \ar[r]^-{\delta^h} \ar[d]_-{\id-\iota\pi}& C^{1, *}(A, \Omega_{\nc}^p(A)) \ar[dl]_-{h_{1, p}}\ar[d]^-{\id-\theta\pi}\ar[r]^-{\delta^h}& C^{2, *}A,\Omega_{\nc}^p(A))\ar[r]^-{\delta^h} \ar[dl]_-{h_{2, p}}\ar[d]^-{\id-\theta\pi}& \cdots \\
0\ar[r] & C^{0,*}(A, \Omega_{\nc}^p(A)) \ar[r]_-{\delta^h} & C^{1,*}(A, \Omega_{\nc}^p(A))\ar[r]_-{\delta^h} & C^{2,*}(A,\Omega_{\nc}^p(A))\ar[r]_-{\delta^h} & \cdots \\
}
\end{equation*}
For any $x=\overline{a_1}\otimes\cdots \otimes \overline{a_p}\otimes a_{p+1}\in C^{0,*}(A, (\sA)^{\otimes p}\otimes A),$ we have \begin{equation*}
\begin{split}
(\id-\iota\circ \pi_{0, p})(x)={} &x-\sum_{i} (-1)^{|f_i|(|x|-|a_1|-1)+ k} \overline{e_i}\otimes\varepsilon(\overline{a_1})\overline{a_2}\otimes \cdots\otimes \overline{a_p}\otimes a_{p+1}f_i
\end{split}
\end{equation*}
and 
\begin{equation*}
\begin{split}
 h_{1, p}\circ \delta^h(x)={}& \sum_i (-1)^{(|f_i|-1)(|x|+k+1)} \overline{e_i}\otimes (\varepsilon\otimes \id^{\otimes p})(\delta^h(x)(\overline{f_i}))\\
={}& -\sum_i (-1)^{(|f_i|-1)(k+1)} \overline{e_i}\otimes (\varepsilon\otimes \id^{\otimes p})(f_i \blacktriangleright x) \\
& +  \sum_i(-1)^{(|f_i|-1)(|x|+k+1)+|x|}\overline{e_i}\otimes 
(\varepsilon\otimes \id^{\otimes p})(x\overline{f_i})\\
={}& x-\sum_{i} (-1)^{|f_i|(|x|-|a_1|-1)+k} \overline{e_i}\otimes\varepsilon(\overline{a_1})\overline{a_2}\otimes \cdots\otimes \overline{a_p}\otimes a_{p+1}f_i,
%={}& x- \iota\circ \pi_{0, p}(x),
\end{split}
\end{equation*}
where, in the third identity, we use Lemma \ref{lemma:bascipropertyeifi1} (2) and $\varepsilon(\overline{a_1}) = 0$ if $|a_1| \neq k$.  
%This shows that $\id-\iota\circ \pi_{0, p} = h_{1, p}\circ \delta^h.$

Similarly, for $m>0$ we have  
\begin{equation*}
\begin{split}
&(\id-\theta_{m-1,p-1}\circ \pi_{m, p})(f)(\overline{a_1}\otimes \cdots \otimes \overline{a_m})\\
={} &f(\overline{a_1}\otimes \cdots \otimes \overline{a_m})-\sum_i (-1)^{|f|(|a_1|-1) +(|f_i|-1)(|f|+k)} \overline{a_1}\otimes e_i\blacktriangleright
(\varepsilon\otimes \id^{\otimes p} )(f(\overline{f_i}\otimes\overline{a_2}\otimes \cdots \otimes \overline{a_m})).
\end{split}
\end{equation*}
On the other hand, we have 
\begin{align*}
&\delta^h\circ h_{m, p}(f) (\overline{a_1}\otimes \cdots \otimes \overline{a_m}) \\
=& -\sum_i (-1)^{(|a_1|-1)(|f|-1)+(|f_i|-1)(|f|+k)} {a_1}\blacktriangleright (\overline{e_i}\otimes (\varepsilon\otimes \id^{\otimes p} )(f(\overline{f_i}\otimes \overline{a_2}\otimes \cdots \otimes \overline{a_m})))\\
&- \sum_i\sum_{j=1}^{m-1} (-1)^{(|f_i|-1)(|f|+k) +|f|+1 + \epsilon_j} \overline{e_i}\otimes (\varepsilon\otimes \id^{\otimes p} )(f(\overline{f_i}\otimes \cdots \otimes \overline{a_ja_{j+1}}\otimes \cdots \otimes \overline{a_m}))\\
&+ \sum_i (-1)^{(|f_i|-1)(|f|+k) +|f|+1  + \epsilon_{m-1}} \overline{e_i}\otimes (\varepsilon\otimes \id^{\otimes p} )(f(\overline{f_i}\otimes \overline{a_1}\otimes \cdots \otimes \overline{a_{m-1}})a_m)
\end{align*}
and 
\begin{align*}
&h_{m+1, p}\circ \delta^h(f)(\overline{a_1}\otimes \cdots \otimes \overline{a_m})\\
={}& -\sum_i (-1)^{(|f_i|-1)(k-1)} \overline{e_i}\otimes (\varepsilon\otimes \id^{\otimes p})(f_i\blacktriangleright f( \overline{a_1}\otimes \cdots \otimes
\overline{a_{m}}))\\
&-\sum_i  (-1)^{|f_i||f|+(|f_i|-1)k} \overline{e_i}\otimes (\varepsilon\otimes \id^{\otimes p})(f(\overline{f_ia_1}\otimes \overline{a_2}\otimes \cdots \otimes \overline{a_{m}}))\\
&-\sum_{j=1}^{m-1}\sum_i (-1)^{|f_i||f|+(|f_i|-1)k+ \epsilon_j}  \overline{e_i}\otimes (\varepsilon\otimes \id^{\otimes p})(f(\overline{f_i}\otimes \cdots \otimes\overline{a_ja_{j+1}}\otimes \cdots \otimes \overline{a_{m}}))\\
&+\sum_i (-1)^{|f_i||f| +(|f_i|-1)k+ \epsilon_{m-1}} \overline{e_i}\otimes (\varepsilon\otimes \id^{\otimes p})(f(\overline{f_i}\otimes \overline{a_1}\otimes \cdots \otimes \overline{a_{m-1}})a_m),
\end{align*}
where $\epsilon_j = |a_1|+\dotsb+|a_j|-j$.
After  cancelling the common  terms in the above two identities, we obtain  
\begin{equation*}
\begin{split}
&(\delta^h\circ h_{m,p}+h_{m+1, p}\circ \delta^h)(f)(\overline{a_1}\otimes \cdots \otimes \overline{a_m})\\
={}& -\sum_i (-1)^{(|a_1|-1)(|f|-1)+(|f_i|-1)(|f|+k)} {a_1}\blacktriangleright (\overline{e_i}\otimes (\varepsilon\otimes \id^{\otimes p} )(f(\overline{f_i}\otimes \overline{a_2}\otimes \cdots \otimes \overline{a_m})))\\
&-\sum_i (-1)^{(|f_i|-1)(k-1)} \overline{e_i}\otimes (\varepsilon\otimes \id^{\otimes p})(f_i\blacktriangleright f( \overline{a_1}\otimes \cdots \otimes
\overline{a_{m}}))\\
&-\sum_i  (-1)^{|f_i||f|+(|f_i|-1)k} \overline{e_i}\otimes (\varepsilon\otimes \id^{\otimes p})(f(\overline{f_ia_1}\otimes \overline{a_2}\otimes \cdots \otimes \overline{a_{m}}))\\
={}&f(\overline{a_1}\otimes \cdots \otimes \overline{a_m})-\sum_i (-1)^{|f|(|a_1|-1) +(|f_i|-1)(|f|+k)} \overline{a_1}\otimes e_i\blacktriangleright
(\varepsilon\otimes \id^{\otimes p} )(f(\overline{f_i}\otimes\overline{a_2}\otimes \cdots \otimes \overline{a_m}))
\end{split}
\end{equation*} 
where in the second equality we use  Lemma \ref{lemma:bascipropertyeifi1} (2).
%verifying the identity (\ref{equation5.6}). %By the definition of $H_{m, p}$ we may conclude that $ \id-\iota\circ \Pi =\delta \circ H +H\circ \delta.$
\end{proof}

For $m, p\in\Z_{>0}$, define $$H_{m, p}\colon C^{m, *}(A, \Omega_{\nc}^p(A))\rightarrow C^{m-1, *}(A, \Omega_{\nc}^p(A))$$ as the composition  $$H_{m, p}:=  \sum_{i=0}^{\min\{p-1,m-1\}}\theta_{m-2, p-1} \circ \cdots\circ \theta_{m-i-1, p-i} \circ h_{m-i, p-i}\circ  \pi_{m-i+1, p-i+1}\circ \cdots \circ \pi_{m, p},$$ 
where the term for $i = 0$ is $h_{m, p}$. For convenience, we set $H_{m, p}=0$ if either $m = 0$ or $p = 0$.

\begin{lemma}
\label{lemma-h}
For $m, p\in \Z_{>0}$ the following diagram commutes
\begin{equation*}
\xymatrix@C=3pc@R=2pc{
C^{m-1, *}(A, \Omega_{\nc}^{p-1}(A))\ar[d]_-{ H_{m-1, p-1}}\ar[r]^-{\theta_{m-1, p-1}}  & C^{m, *}(A, \Omega_{\nc}^p(A))\ar[d]^-{ H_{m, p}}\\
C^{m-2, *}(A, \Omega_{\nc}^{p-1}(A)) \ar[r]^-{\theta_{m-2, p-1}}  & C^{m-1, *}(A, \Omega_{\nc}^p(A)).
}
\end{equation*}
\end{lemma}

\begin{proof}
We have that
\begin{align*}
&H_{m, p}\circ \theta_{m-1, p-1}\\
={}& \sum_{i=0}^{\min\{p-1,m-1\}}\theta_{m-2, p-1} \circ \cdots\circ \theta_{m-i-1, p-i} \circ h_{m-i, p-i}\circ  \pi_{m-i+1, p-i+1}\circ \cdots \circ \pi_{m, p}\circ \theta_{m-1, p-1}\\
%={}&\sum_{i=1}^{\min\{p-1,m-1\}}\theta_{m-2, p-1} \circ \cdots\circ \theta_{m-i-1, p-i} \circ h_{m-i, p-i}\circ  \pi_{m-i+1, p-i+1}\circ \cdots \circ \pi_{m-1, p-1}\\
={}&\theta_{m-2, p-1}\circ H_{m-1, p-1},
\end{align*}
where the second identity follows from Lemma \ref{lemma5.3} and the fact that $h_{m, p} \circ \theta_{m-1, p-1} = 0$. 
 \end{proof}
The above lemma allows us to make the following definition. 
\begin{definition}
\label{defn-H}
Define $H\colon \calC_{\sg}^*(A, A) \rightarrow  \calC_{\sg}^{*-1}(A, A)$ to be the map induced by the maps $H_{m,p}$ after passing to the colimit. That is, for any $\widetilde f \in \calC_{\sg}^*(A, A)$ which is represented by $f \in C^{m, *}(A, \Omega_{\nc}^p(A))$ for some $m, p$, we define
$$
H(\widetilde f) = H_{m, p}(f) \in \calC_{\sg}^{*-1}(A, A).
$$ 
\end{definition}

\begin{proof}[Proof of Theorem \ref{thm:homtopyretract}] 
It follows directly from Lemmas \ref{lemma:pimp} and \ref{lemma-h} that the map $H$ is a chain homotopy between $\id$ and $\iota\circ \Pi$. Namely, $$\id-\iota\circ \Pi =\delta_{\sg} \circ H +H \circ \delta_{\sg}. $$
\end{proof}

 \bibliographystyle{plain}

\begin{thebibliography}{99}

\bibitem{Abb1} H. Abbaspour, {On algebraic structures of the Hochschild complex}, \textit{Free loop spaces in geometry and topology}, 165--222, IRMA  Lect. Math. Theor. Phys.  24, Eur. Math. Soc. Z\"urich (2015).

 \bibitem{Abb}
H. Abbaspour, 
{On the Hochschild homology of open Frobenius algebras,}
\textit{J. Noncommut. Geom.} 10 (2016), no. 2, 709-743.

\bibitem{Buc} R.O. Buchweitz, {Maximal Cohen-Macaulay modules and Tate-cohomology over Gorenstein rings}, http:hdl.handle.net/1807/16682 (1986).

\bibitem{ChSu}
M. Chas and D. Sullivan, 
{String topology,}
preprint, arXiv:math/9911159. 

\bibitem{CiFrOa}
K. Cieliebak, U. Frauenfelder and A. Oancea,
{Rabinowitz Floer homology and symplectic homology,} 
\textit{Ann. Sci. \'Ec. Norm. Sup\'er.} (4) 43 (2010), no. 6, 957-1015.

\bibitem{CoJo}
R. Cohen and J. Jones,
{A homotopy theoretic realization of string topology,} \textit{Math. Ann.}
324(4) (2002) 773-798.

\bibitem{DrPoRo} 
G. Drummond-Cole, K. Poirier, and N. Rounds, 
{Chain level string topology operations,} preprint,  arXiv:1506.02596.


\bibitem{GoHi}
M. Goresky and N. Hingston,
{Loop products and closed geodesics,}
\textit{Duke Math. J.} 150 (2009), no. 1, 117-209.

\bibitem{HiWa}
N. Hingston and N. Wahl,
{Product and coproduct in string topology,} preprint, arXiv:1709.06839.

\bibitem{HiWa2}
N. Hingston and N. Wahl,
{On the invariance of the string topology coproduct,} preprint, arXiv:1908.03857v3.

\bibitem{Kla}
A. Klamt, 
{Natural operations on the Hochschild complex of commutative Frobenius algebras via 
the complex of looped diagrams,} preprint,
arXiv:1309.4997.

\bibitem{Kau1}
R. Kaufmann, 
{Moduli space actions on the Hochschild co-chains of a Frobenius algebra.II. Correlators.}
\textit{J. Noncommut. Geom.} 2 (2008), no. 3, 283--332

\bibitem{Kau}
R. Kaufmann,
{A detailed look on actions on Hochschild complexes especially the degree $1$ co-product and actions on loop spaces,} preprint, arXiv:1807.10534. 

\bibitem{Kel}  B. Keller, {Singular Hochschild cohomology via the singularity category},
\textit{C. R. Math. Acad. Sci. Paris} 356  (11-12) (2018), 1106--1111.


 \bibitem{LaSt} P. Lambrechts and D. Stanley,
{Poincar\'e duality and commutative differential graded algebras,}
\textit{Ann. Sci. \'Ec. Norm. Sup\'er.} (4) 41 (2008), no. 4, 495-509.

\bibitem{Na} F. Naef,
{The string coproduct "knows" Reidemeister/Whitehead torsion}, preprint, arXiv:2106.11307.

\bibitem{NaWi}
F. Naef and T. Willwacher, 
{String topology and configuration spaces of two points}, preprint, arXiv:1911.06202.

\bibitem{Orl} 
D. Orlov, 
{Triangulated categories of singularities and $D$-branes in Landau-Ginzburg models},  \textit{Proc. Steklov Inst. Math.} 246 (3) (2004), 227--248.


\bibitem{RiWa}
M. Rivera and Z. Wang,
{Singular Hochschild cohomology and algebraic string operations}, 
\textit{J. Noncommut. Geom.} 13 (2019), 297-361.

\bibitem{Sul}
D. Sullivan,
{Open and closed string field theory interpreted in classical algebraic topology},  
{\it Topology, geometry and quantum field theory}, London Math. Soc. Lecture Note Ser. 308 (2004) 344-357. Cambridge Univ. Press.


\bibitem{TrZe}
T. Tradler and M. Zeinalian, 
{Algebraic string operations,} \textit{K-Theory} 38 (2007), 59-82.

\bibitem{VdB}
M. Van den Bergh, 
{Existence theorems for dualizing complexes over non-commutative graded and filtered rigns,} \textit{J. Algebra} 195 (1997), 662-679.


\bibitem{WaWe}
N. Wahl and C. Westerland,
{Hochschild homology of structured algebras,}
\textit{Adv. Math.} 288 (2016), 240-307.

\bibitem{Wan}
Z. Wang,
{Gerstenhaber algebra and Deligne's conjecture on Tate-Hochschild cohomology,} \textit{Trans. Amer. Math. Soc.} 374 (2021), 4537-4577.
%arXiv:1801.07990.

\end{thebibliography}

\Addresses
\end{document}